\documentclass[a4paper, 12pt]{article}

\usepackage[top=2.0cm,bottom=2.5cm,left=2.0cm,right=2.0cm]{geometry}
\usepackage{lineno,hyperref}
\usepackage{amsmath,amsthm,amsfonts,amssymb,mathrsfs,graphicx,color}
\usepackage{subfigure,float}
\usepackage{enumerate}
\usepackage[T1]{fontenc}
\usepackage[title]{appendix}

\modulolinenumbers[5]

\begin{document}
\newtheorem{definition}{Definition}[section]
\newtheorem{theorem}[definition]{Theorem}
\newtheorem{lemma}[definition]{Lemma}
\newtheorem{example}[definition]{Example}
\newtheorem{remark}[definition]{Remark}
\newtheorem{corollary}[definition]{Corollary}
\newtheorem{proposition}[definition]{Proposition}

\title{The net-regular strongly regular signed graphs with degree 6}

\author{Qian Yu\thanks{1779382456@qq.com}, Yaoping Hou\thanks{Corresponding author: yphou@hunnu.edu.cn}\\
{\small MOE-LCSM, CHP-LCOCS, School of Mathematics and Statistics}\\{\small Hunan Normal University, Changsha, Hunan 410081, P. R. China}
}

\date{}
\maketitle

\begin{abstract}
In this paper, we study the net-regular strongly regular signed graphs with degree 6 and determine all connected 6-regular and net-regular strongly regular signed graphs. There are three, six and four 6-regular strongly regular signed graphs with net-degree 4, 2 and 0, respectively.
\end{abstract}


\section{Introduction}

Let $G$ be a simple connected graph with vertex set $V(G)$ and edge set $E(G)$. A \emph{signed graph} $\dot{G}=(G,\sigma)$ is an underlying graph $G$ together with a sign function $\sigma :E(G)\rightarrow\{+1,-1\}$. The \emph{adjacency matrix} of $\dot{G}$ is denoted by $A(\dot{G})=(a_{ij}^\sigma)$ where $a_{ij}^\sigma=\sigma(v_iv_j)$ if $v_i\sim v_j$, and 0 otherwise. And the eigenvalues of $A(\dot{G})$ are the eigenvalues of $\dot{G}$. Let $N(v)$ denote the set of neighbours of $v$ and  $N(u)\bigcap N(v)$ denote the set of common neighbours of $u$ and $v$. The \emph{degree} $d(v)$ of a vertex $v$ in $\dot{G}$ is the number of edges incident with $v$. And the \emph{positive degree} $d^{+}(v)$ and \emph{negative degree} $d^{-}(v)$ of $v$ are the number of positive and negative edges incident with $v$, respectively. We use $v_i\overset{+}{\sim}v_j$ and $v_i\overset{-}{\sim}v_j$ to denote that $v_i$ is positively and negatively adjacent to $v_j$, respectively. Let $N^+(v)$ ($N^-(v)$, respectively) denote the set of positive (negative, respectively) neighbours of $v$. A signed graph  $\dot{G}$ is said to be \emph{$r$-regular} if its every vertex has the same degree $r$ and \emph{$\rho$ net-regular} if its every vertex has the same net-degree $\rho$.   A signed graph $\dot{G}$ is called \emph{homogeneous} if all its edges have the same sign. Otherwise, it is \emph{inhomogeneous}. We use $G^+$ and  $G^-$ to denote the subgraphs of $G$ induced by the positive edges and negative edges of $\dot{G}$ in this paper. And $-\dot{G}$ denotes the \emph{negation} of $\dot{G}$.
%

The sign of a cycle or a walk in a signed graph is the product of the sign of its edges. A signed graph is \emph{balanced} if all its cycle is positive. Otherwise, it is  \emph{unbalanced}.

It is  known that a strongly regular graph with parameters $(n,r,e,f)$ is a $r$-regular graph on $n$ vertices in which any two adjacent vertices have exactly $e$ common neighbours and any two non-adjacent vertices have exactly $f$ common neighbours \cite{CRS}. And the parameters of a strongly regular graph satisfy that
\begin{equation}\label{srg}
   r(r-e-1)=(n-r-1)f.
\end{equation}

There are some different definitions on the strong regularity of signed graphs.  Zaslavsky  defined the very strongly regular signed graphs in \cite{Z}. Ramezani gave a kind of concept of signed strongly regular graphs and  constructed some families of signed strongly regular graphs with only two distinct eigenvalues by the star complement technique \cite{R}. But we focus on the following definition given by Stani\'c in  \cite{S1}.

By \cite{S1}, a signed graph $\dot{G}$ is called strongly regular if it is neither homogeneous complete nor edgeless, and there exist $r\in \mathbb{N}$, $a, b ,c\in\mathbb{Z}$, such that the entries of $A^2_{\dot{G}}$ satisfy

\begin{equation*}
  a^{(2)}_{ij}=
\left\{
\begin{aligned}
r, & \quad \text{if $v_i=v_j$}; \\
a, & \quad \text{if $v_i\overset{+}{\sim}v_j$}; \\
b, & \quad \text{if $v_i\overset{-}{\sim}v_j$}; \\
c, & \quad \text{if $v_i\neq v_j$ and $v_i\nsim v_j$}.
\end{aligned}
\right.
\end{equation*}
By the definition, a signed graph $\dot{G}$ is strongly regular if and only if $A^2_{\dot{G}}$ satisfies
\begin{equation}\label{def}
  A^2_{\dot{G}}+\frac{b-a}{2}A_{\dot{G}}=\frac{a+b}{2}A_G+cA_{\overline{G}}+rI=\frac{a+b-2c}{2}A_G+cJ+(r-c)I.
\end{equation}
where $\overline{G}$ is the complement of $G$, $I$ and $J$ are the identity matrix and the all-1 matrix, respectively.

If $\dot{G}$ is $\rho$ net-regular, then the all-1 vector $\mathbf{j}$ is the eigenvector of  $\dot{G}$, $G$ and $\overline{G}$. So we have

\begin{equation}\label{abcrn}
  \rho^2+\frac{b-a}{2}\rho=\frac{a+b}{2}r+c(n-r-1)+r.
\end{equation}

In \cite{KS}, Koledin and Stani\'{c} classified all inhomogeneous strongly regular signed graphs (we write SRSGs shortly in further text) into the following five disjoint classes:

$\mathcal{C}_1$: SRSGs with $a=-b$, which are either complete or non-complete with $c\neq0$.

$\mathcal{C}_2$: SRSGs with $a=-b$, which are non-complete with $c=0$.

$\mathcal{C}_3$: SRSGs with $a\neq-b$, which are either complete or non-complete with $c=\frac{a+b} {2}$ .

$\mathcal{C}_4$: SRSGs with  $a\neq-b$, which are non-complete with $c=0$.

$\mathcal{C}_5$: SRSGs with  $a\neq-b$, which are non-complete with $c\neq\frac{a+b} {2}$ and $c\neq0$.

There are some results about this kind of strongly regular signed graphs. Z. Stani\'c also gave certain structural and spectral properties of such signed graphs in \cite{S1}. In particular, Stani\'c characterized the bipartite SRSGs and the SRSGs with 4 eigenvalues. In \cite{KS}, T. Koledin and Z. Stanić studied the SRSGs in class $\mathcal{C}_3$.   M. An\dj eli\'{c}, T. Koledin and Z.  Stani\'c determined the net-regular strongly regular signed graphs with $r\leq4$ in \cite{AKS1}. And they considered  relationships between net-regular, strongly regular and walk-regular signed graphs in \cite{AKS}.

We have studied the net-regular strongly regular signed graphs with $r=5$ in \cite{YH}. Therefore, we focus on the connected net-regular strongly regular signed graphs with degree 6 in this paper. Two known results about strongly regular signed graphs are given in Section 2. In Section 3, we determine all the connected net-regular strongly regular signed graphs with degree 6. And it is worth to note that the positive and negative edges of signed graphs in figures are represented by solid and dashed lines, respectively, in this paper.

\section{Preliminaries}

In this section, we give two  Lemmas which will be used in next section. The first one examines the relationship between parameters $a,b$ of an inhomogeneous strongly regular signed graph and its negation. The second one characterizes the number of negative walks of length 2 between two vertices in a connected non-complete net-regular strongly regular signed graph.

\begin{lemma}\cite{KS}\label{le1}
 If G is a SRSG belonging to $\mathcal{C}_i$, $1\leq i\leq 5$, so is $-\dot{G}$. The parameters $a$ and $b$ of an inhomogeneous SRSG $\dot{G}$ interchange their roles in the set of parameters of $-\dot{G}$.
\end{lemma}

\begin{lemma}\cite{AKS1}\label{le2}
  Let $\dot{G}$ be a connected non-complete net-regular SRSG belonging to $\mathcal{C}_1\cup \mathcal{C}_4\cup \mathcal{C}_5$. Then for any two vertices $v_i$ and $v_j$ of $\dot{G}$, the number of negative walks of length 2 between $v_i$ and $v_j$ is even. Moreover, if this number is $2k$, $k\in \mathbb{N}$, then there are exactly $k$ common neighbours of $v_i$ and $v_j$ that are joined to $v_i$ by a positive edge and to $v_j$ by a negative edge, and exactly $k$ common neighbours that are joined to $v_i$ by a negative edge and to $v_j$ by a positive edge.
\end{lemma}

\section{The net-regular SRSGs with degree 6}

 In this section, we consider the net-regular SRSGs with degree 6. And we only need to consider the non-negative net-degree, i.e. 4, 2 and 0, by Lemma \ref{le1}.

\subsection{The 6-regular SRSGs with net-degree 4}
Let $\dot{G}$ be a  6-regular SRSG with net-degree 4, then $G^+$ is 5-regular and $G^-$ is 1-regular. Therefore, $G^-$  deduces a perfect matching of $\dot{G}$ and so $n$ is even.

If $\dot{G}\in \mathcal{C}_2$ or $\dot{G}$ is a complete signed graph in $\mathcal{C}_1$, then we have $A^2_{\dot{G}}+bA_{\dot{G}}=6I$ by equation \eqref{def}. It is obvious that $\dot{G}$ has two eigenvalues: net-degree 4 and $\frac{-3}{2}$. This contradicts that $b$ is integral.

If $\dot{G}\in \mathcal{C}_3$, then $G$ is isomorphic to $K_7$ \cite{KS},  which is contradictory to that $n$ is even.

Now we consider  the non-complete $\dot{G}\in \mathcal{C}_1\bigcup\mathcal{C}_4\bigcup\mathcal{C}_5$. And we have

\begin{equation}\label{6-4-abcn}
 10-5a-b=c(n-7)
\end{equation}
by equation \eqref{abcrn}.

Firstly, we give the following Lemma to show that $\dot{G}$ doesn't contain unbalanced triangles.

\begin{lemma}\label{triangle-1}
Let $\dot{G}\in \mathcal{C} _1\bigcup\mathcal{C}_4\bigcup\mathcal{C}_5$ be a connected and non-complete 6-regular SRSG with net-degree 4. then it contains no unbalanced triangle.
\end{lemma}

\begin{proof}
If  $\dot{G} $ contains an unbalanced triangle, then this  triangle must have only one negative edge since vertex net-degree is 4. Suppose $v_iv_jv_k$ is an unbalanced triangle with $v_i\overset{+}{\sim}v_j$, $v_i\overset{-}{\sim}v_k$, $v_j\overset{+}{\sim}v_k$. Then there is another common neighbour $v_l$ of $v_i,v_j$ such that $v_l\overset{+}{\sim}v_i$, $v_l\overset{-}{\sim}v_j$ by Lemma \ref{le2} and so we have $a\leq 1$. We're going to do this in following two cases since $a=1$ might imply that there exist two positively vertices have only one common neighbour.

If $a=1$, then there are other three common positive neighbours $v_m,v_s,v_t$ for $v_i$ and $v_j$. We have $v_l\overset{+}{\sim}v_k,v_m,v_t,v_s$ and $v_k\overset{+}{\sim}v_m,v_t,v_s$ by considering the positive edges $v_iv_l$ and $v_jv_k$. Now vertices $v_i,v_j,v_k,v_l$ have enough degree. We consider the positive edge $v_iv_t$, then we have $v_t\overset{-}{\sim}v_s$ or $v_t\overset{-}{\sim}v_m$. Without loss of generality, we suppose that $v_t\overset{-}{\sim}v_s$. Then $v_t\overset{+}{\sim}v_m$. And we have $v_s\overset{+}{\sim}v_m$ by positive edge $v_iv_s$. But the net-degree of vertex $v_m$ will be 6. A contradiction.

If $a<0$, then there must be two negative walks of length 2 between $v_i$ and its every positive neighbour. So every positive neighbour of $v_i$ is positively adjacent to $v_k$ and it need to be negatively adjacent to another one positive neighbour of $v_i$. But there are five positive neighbours for $v_i$, they can not be paired.

If $a=0$, then there may be zero or four common neighbours between $v_i$ and its positive neighbours. Suppose $N(v_i)\bigcap N(v_j)=\{v_k,v_l,v_m,v_s\}$ and $v_{t_1}$, $v_{t_2}$ are the remaining neighbours of $v_i$ and $v_j$, respectively. Then $v_m,v_s\overset{+}{\sim}v_i,v_j$ and $v_{t_1}\overset{+}{\sim}v_i$, $v_{t_2}\overset{+}{\sim}v_j$. For the edge $v_iv_{t_1}$, if $v_{t_1}$ has common neighbours with $v_i$, then there must be two negative walks of length 2 between $v_i$ and $v_{t_1}$ by $a=0$. In this case, there are two negative walks of length 2 between $v_i$ and its positive neighbours. Then every positive neighbour of $v_i$ is positively adjacent to $v_k$. And all  positive neighbours of $v_i$ need to be paired with negative edges. This is impossible obviously. So $v_{t_1}$ don't have common neighbours with $v_i$ and $v_{t_2}$ don't have common neighbours with $v_j$. Then we have $v_l\overset{+}{\sim}v_k,v_m,v_s$  and  $v_k\overset{+}{\sim}v_m,v_s$ by the positive edges $v_iv_l$ and $v_jv_k$ . So we get that $b=4$ by the negative edge $v_jv_l$ since $v_{t_2}\nsim v_j$. And then we have  $v_m\overset{-}{\sim}v_s$ by positive edge $v_iv_m$.

By equation \eqref{6-4-abcn}, we have $n=13,10,9,8$ if  $(a,b)=(0,4)$. Since $n$ is even,  $n=10,8$. In fact, we have $n>12$ since  $v_{t_1}$ has no common neighbours with $v_i$.

Therefore,  $\dot{G} $  contains no unbalanced triangles.
\end{proof}

If  $\dot{G}\in \mathcal{C} _1\bigcup\mathcal{C}_4\bigcup\mathcal{C}_5$ is a non-complete 6-regular SRSG with net-degree 4, we have that any two negatively adjacent vertices don't have common neighbours, i.e. $b=0$, by above Lemma and $G^-$ is 1-regular. And so $a\in \{0,1,2,3,4\}$. In this case, we get that
\begin{equation*}
  10-5a=c(n-7)
\end{equation*}
by equation \eqref{6-4-abcn} and $b=0$.

Therefore, the possible parameters sets are :
 $(17,6,0,0,1)$, $(12,6,0,0,2)$, $(9,6,0,0,5)$, $(12,6,1,0,1)$, $(8,6,1,0,5)$, $(n,6,2,0,0)$, $(7,6,2,0)$ (Complete), $(12,6,3,0,-1)$, $(8,6,3,0,-5)$, $(17,6,4,0,-1)$, $(12,6,4,0,-2)$, $(9,6,4,0,-5)$.

Since $n$ is even, the cases of $n=17,9,7$ are taken out.
Let $v_iv_j$ be a negative edge of $\dot{G}$, then there is no common neighbour between $v_i$ and $v_j$.  Both $v_i$ and $v_j$ have another five positive neighbours, respectively, by $r=6$.  So we have $n\geq12$.

If $\dot{G}$ has parameters $(12,6,1,0,1)$, then there is only one positive walk of length 2 between two positively adjacent vertices. But the five  positive neighbours of $v_i$ cannot be paired. A contradiction.

For parameters $(12,6,0,0,2)$: $a=b=0$ implies two adjacent vertices don't have common neighbours and $c=2$ implies that there are only two positive walks of length 2 or four positive walks and two negative walks of length 2 between two non-adjacent vertices. Suppose $V(\dot{G})=\{v_1,v_2,\dots,v_{12}\}$ and $N^+(v_1)=\{v_2,\dots,v_6\}$, $N^-(v_1)=\{v_7\}$, $N^+(v_7)=\{v_8,\dots,v_{12}\}$. We have $v_7\nsim v_2,\dots,v_6$ by $a=b=0$. Since there is one  negative walk of length 2 between $v_7$ and $v_2,\dots,v_6$, respectively, $|N(v_7)\bigcap N(v_2)|=|N(v_7)\bigcap N(v_3)|=\cdots=|N(v_7)\bigcap N(v_6)|=6$. Therefore, $G=K_{6,6}$. So we get that $\dot{G}=\dot{S}^1_{12}$ (Figure \ref{S12-1}) since every vertex has net-degree 1.

\begin{figure}[h]
  \centering
  \includegraphics[width=7cm]{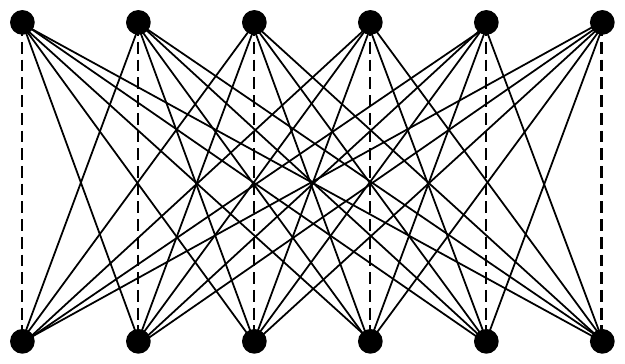}
  \caption{$\dot{S}^1_{12}$}\label{S12-1}
\end{figure}

If $\dot{G}$ has parameters  $(12,6,3,0,-1)$, then there are only three positive walks of length 2 between two positively adjacent vertices. Suppose $v_iv_j$ is a positive edge of $\dot{G}$,  $N(v_i)\bigcap N(v_j)=\{v_k,v_l, v_m\}$ and $v_s\overset{+}{\sim}v_i$, $v_t\overset{+}{\sim}v_j$. Then we have  $v_s, v_t\overset{+}{\sim}v_k,v_l, v_m$ by $a=3$. We consider the positive edge $v_iv_k$, then  $v_k\overset{+}{\sim}v_l$ or $v_k\overset{+}{\sim}v_m$. But there are only two positive walks of length 2 between $v_i,v_m$ or  $v_i,v_l$. A contradiction.

If $\dot{G}$ has parameters $(12,6,4,0,-2)$, then $a=4$ implies that $G^+$ is the union  of $K_6$, so $\dot{G}$ is isomorphic to  $\dot{S}^2_{12}$  as shown in Figure \ref{S12-2}.

\begin{figure}
  \centering
  \includegraphics[width=6cm]{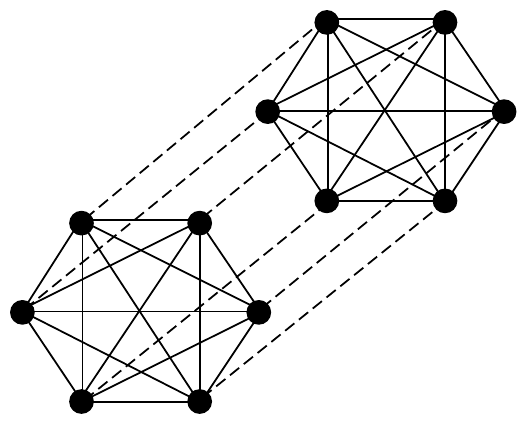}
  \caption{$\dot{S}^2_{12}$}\label{S12-2}
\end{figure}

If $\dot{G}$ has parameters $(n,6,2,0,0)$, then there are only two positive walks of length 2 between two positively adjacent vertices in $\dot{G}$. Suppose  $v_1v_2$ is a negative edge of $\dot{G}$, then $|N(v_1)\bigcap N(v_2)|=0$. So we suppose $N^+(v_1)=\{v_3,v_4, v_5,v_6,v_7\}$ and $N^+(v_2)=\{v_8,v_9, v_{10},v_{11},v_{12}\}$. Consider the positive edge $v_1v_3$, then $v_3$ is positively adjacent to two of $\{v_4,v_5,v_6,v_7\}$ by $a=2$. Without loss of generality, we suppose $v_3\overset{+}{\sim}v_4, v_7$. Then $v_4\overset{+}{\nsim}v_7$. Otherwise, positive edges $v_1v_5$ and $v_1v_6$ cannot satisfy $a=2$. Then $v_4$ is positively adjacent to  one of $\{v_5,v_6\}$. We suppose $v_4\overset{+}{\sim}v_5$. Now, $v_5$ need to be positively adjacent to  one of $\{v_7,v_6\}$. In fact,  we have $v_5\overset{+}{\nsim}v_7$. Otherwise, positive edge $v_1v_6$ cannot satisfy $a=2$. So we get that $v_5\overset{+}{\sim}v_6$ and $v_6\overset{+}{\sim}v_7$. Similarly, we have $v_8\overset{+}{\sim}v_9$, $v_9\overset{+}{\sim}v_{10}$, $v_{10}\overset{+}{\sim}v_{11}$, $v_{11}\overset{+}{\sim}v_{12}$ and $v_{12}\overset{+}{\sim}v_8$.

Now we consider the two non-adjacent vertices $v_1$ and $v_8$. They need one negative walk and two positive walks of length 2 by $c=0$. So we can suppose $v_3\overset{-}{\sim}v_8$. And we can also suppose $v_4\overset{-}{\sim}v_9$, $v_5\overset{-}{\sim}v_{10}$, $v_6\overset{-}{\sim}v_{11}$ and $v_7\overset{-}{\sim}v_{12}$ for the non-adjacent vertices pairs $\{v_1,v_9\}$, $\{v_1,v_{10}\}$, $\{v_1,v_{11}\}$ and $\{v_1,v_{12}\}$ without loss of generality. Also by vertices $v_1$ and $v_8$, we have $v_5, v_6 \overset{+}{\sim}v_8$ since $v_3\overset{+}{\sim}v_4, v_7$ and $v_3\overset{-}{\sim}v_8$.  In a similar way, we have $v_6, v_7 \overset{+}{\sim}v_9$,  $v_3, v_7 \overset{+}{\sim}v_{10}$, $v_3, v_4 \overset{+}{\sim}v_{11}$ and $v_4, v_5 \overset{+}{\sim}v_{12}$ for the non-adjacent vertices pairs $\{v_1,v_9\}$, $\{v_1,v_{10}\}$, $\{v_1,v_{11}\}$ and $\{v_1,v_{12}\}$. Now, every vertex has degree 6. Then we get a signed graph  $\dot{S}^3_{12}$ in Figure \ref{6-S12}. It is a SRSG with parameters $(12,6,2,0,0)$.

\begin{figure}[h]
  \centering
  \includegraphics[width=8cm]{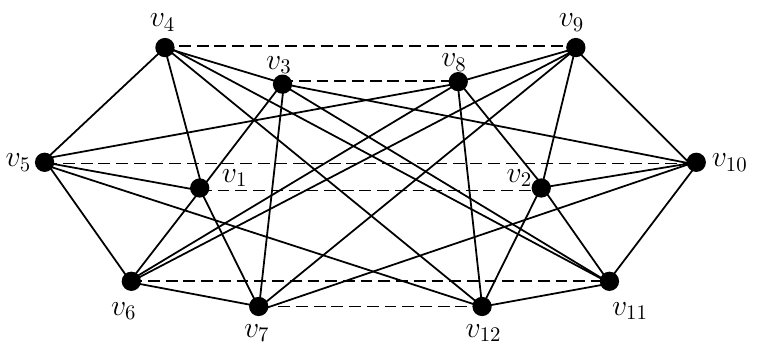}
  \caption{$\dot{S}^3_{12}$}\label{6-S12}
\end{figure}

To sum up, the following theorem can be obtained.

\begin{theorem}
  There are three connected 6-regular and 4 net-regular SRSGs:  $\dot{S}^1_{12}$, $\dot{S}^2_{12}$  and $\dot{S}^3_{12}$.
\end{theorem}

\subsection{The 6-regular SRSGs with net-degree 2}

Let $\dot{G}$ be a connected  6-regular and 2 net-regular SRSG, then $G^+$ is 4-regular and $G^-$ is 2-regular. Therefore,  $G^-$ is a cycle or the union of cycles.

If $\dot{G}\in \mathcal{C}_1$ is complete or $\dot{G}\in \mathcal{C}_2$, then $\dot{G}$ has two eigenvalues 2 and $-3$ by $A^2_{\dot{G}}+bA_{\dot{G}}-6I=0$. Therefore, we have $b=1$ and $a=-1$. So the parameters are $(n,6,-1,1,0)$ or $(7,6,-1,1)$.
 In fact, we have $(a,b)\neq(-1,1)$ by the following Lemma.

\begin{lemma}\label{6-2-a-1b1}
 If $\dot{G}$ is a connected non-complete 6-regular and  2 net-regular SRSG, then  $(a,b)\neq(-1,1)$.
\end{lemma}

\begin{proof}
  We suppose that $(a,b)=(-1,1)$ for $\dot{G}$.   Let $v_iv_j$ be an edge of  $\dot{G}$ and $\sigma$ be the sign function of $\dot{G}$. If $\sigma(v_iv_j)=-1$, then $v_i,v_j$ have one or five common neighbours by $b=1$. We claim that $v_i,v_j$ must have one common neighbour for  $\sigma(v_iv_j)=-1$. Otherwise, suppose that $N(v_i)\bigcap N(v_j)= \{v_k,v_l,v_m,v_s,v_t\}$ and $v_k,v_m,v_t\overset{+}{\sim}v_i,v_j$,   $v_l\overset{+}{\sim}v_i$,   $v_l\overset{-}{\sim}v_j$, $v_s\overset{-}{\sim}v_i$,  $v_l\overset{+}{\sim}v_j$. We consider the negative edge $v_iv_s$. Since the situation of the two negative walks  of length 2 between $v_i$ and $v_k$ is contradictory to Lemma \ref{le2} if $v_s\overset{+}{\sim} v_k$, we have  $v_s\overset{-}{\sim} v_k$. Then we have  $v_s\overset{+}{\sim}v_l,v_m,v_t$ by $b=1$. But the situation of the two negative walks of length 2 between $v_i$ and $v_m$ also contradicts Lemma \ref{le2}.

  Therefore, two negatively adjacent vertices have one common neighbour and two positively adjacent vertices have three common neighbours by  $(a,b)=(-1,1)$. Suppose that  $\sigma(v_iv_j)=1$, $N(v_i)= \{v_j, v_k,v_l,v_m,v_{s_1},v_{t_1}\}$ and $N(v_j)= \{v_i, v_k,v_l,v_m,v_{s_2},v_{t_2}\}$. We say that the edges of the positive walk of length 2 between $v_i$ and $v_j$ are positive. Otherwise, there exist two negatively adjacent vertices have  one negative walk of length 2. Then we can suppose that $v_k\overset{+}{\sim}v_i$,   $v_k\overset{-}{\sim}v_j$, $v_l\overset{+}{\sim}v_i,v_j$, $v_m\overset{-}{\sim}v_i$,  $v_m\overset{+}{\sim}v_j$ and $v_{s_1}\overset{-}{\sim}v_i,v_{t_1}\overset{+}{\sim}v_i$. We get that  $v_m$ is not adjacent to other neighbours of $v_i$ by the negative edge $v_i v_m$. Then $v_k\overset{+}{\sim}v_{s_1}$ and $v_k\overset{+}{\sim}v_l$ or $v_k\overset{+}{\sim}v_{t_1}$ by  $a=-1$. If  $v_k\overset{+}{\sim}v_l$, there will be  two positive walks of length 2 between $v_i$ and $v_l$. This is contradictory to $a=-1$. For $v_k\overset{+}{\sim}v_{t_1}$, we get that $a\geq 1$ for the positive edge $v_iv_{t_1}$ since $v_{t_1}\nsim v_{s_1}, v_m$. A contradiction.
\end{proof}

If $\dot{G}\in \mathcal{C}_3$, then $\dot{G}$ is complete \cite{KS}. So $G^-$ is $C_7$ or $C_3\bigcup C_4$. But it is easy to verify that the corresponding signed graphs are not SRSGs.

Now we consider the non-complete $\dot{G}\in \mathcal{C}_1\bigcup\mathcal{C}_4\bigcup\mathcal{C}_5$. We have
\begin{equation}\label{6-2-abcn}
4a+2b+c(n-7)=-2
\end{equation}
by equation \eqref{abcrn}.

At first, we give some conditions for the parameters $a$ and $b$ in the following Lemma.

\begin{lemma}\label{6-2-a}
  Let $\dot{G}\in \mathcal{C}_1\bigcup\mathcal{C}_4\bigcup\mathcal{C}_5$ be a connected non-complete 6-regular and 2 net-regular SRSG. Then

  (1) $a\neq-3$;

  (2) If $a=-4$, then $\dot{G}=\dot{S}^2_8$ (Figure \ref{86-446}) has parameters $(8,6,-4,4,6)$;

  (3) $b\neq5$;

  (4) If  $b=4$, it must have $a=-4$.
\end{lemma}

\begin{proof}

  (1) Suppose that  $v_iv_j$ is a positive edge of  $\dot{G}$. If $a=-3$, then there are four negative walks and one positive walk of length 2 between two positively adjacent vertices. We suppose $N(v_i)\bigcap N(v_j)=\{v_k, v_l, v_m, v_s, v_t\}$ and $v_k,v_l\overset{+}{\sim}v_i$, $v_k,v_l\overset{-}{\sim}v_j$ and $v_m,v_s\overset{-}{\sim}v_i$, $v_m,v_s\overset{+}{\sim}v_j$,   $v_t\overset{+}{\sim}v_i,v_j$.  By the positive edges $v_iv_k$ and  $v_iv_l$, we have $v_k\sim v_l,v_m,v_s,v_t$ and $v_l\sim v_m,v_s,v_t$. And by the positive edges $v_jv_m$ and  $v_jv_s$, we have $v_m\sim v_s,v_t$ and $v_s\sim v_t$. Then  $\dot{G}$ is complete. A contradiction.

  (2) Suppose that  $v_iv_j$ is a positive edge of  $\dot{G}$. If $a=-4$, then $v_i$, $v_j$ have four common neighbours which denoted by $v_k$, $v_l$, $v_m$, $v_s$. We suppose that $v_k,v_l\overset{+}{\sim}v_i$, $v_k,v_l\overset{-}{\sim}v_j$ and $v_m,v_s\overset{-}{\sim}v_i$, $v_m,v_s\overset{+}{\sim}v_j$. Let $v_{t_1}$,  $v_{t_2}$ be the remaining neighbour of $v_i$, $v_j$, respectively. Then $v_{t_1}\overset{+}{\sim}v_i$ and $v_{t_2}\overset{+}{\sim}v_j$. Since $a=-4$, we have $v_{t_1}\overset{-}{\sim}v_k,v_l$, $v_{t_1}\overset{+}{\sim}v_m,v_s$ and  $v_{t_2}\overset{-}{\sim}v_m,v_s$, $v_{t_2}\overset{+}{\sim}v_k,v_l$ by the positive edges  $v_i v_{t_1}$ and $v_j v_{t_2}$. And by the positive edges  $v_i v_k$ and $v_i v_l$, we must have $v_k,v_l\overset{+}{\sim}v_m,v_s$. Now, $d(v_k)=d(v_l)=d(v_m)=d(v_s)=6$. There are only three negative walks of length 2 between $v_m$ and $v_{t_1}$. So we have $v_{t_1}\overset{+}{\sim}v_{t_2}$. Now $\dot{G}=\dot{S}^2_8$ is a SRSG  with parameters $(8,6,-4,4,6)$.

  (3) Suppose that  $v_iv_j$ is a negative edge of  $\dot{G}$. If $b=5$, then two negatively adjacent vertices have five common neighbours.  We suppose $N(v_i)\bigcap N(v_j)=\{v_k, v_l, v_m, v_s, v_t\}$ and  $v_k,v_l,v_m, v_s\overset{+}{\sim}v_i, v_j$,  $v_t\overset{-}{\sim}v_i,v_j$. By the negative edge $v_jv_t$, we have $v_t\overset{+}{\sim} v_k,v_l,v_m, v_s$. For the positive edge $v_iv_k$, we get that $a=-4$ or $-3$ by Lemma \ref{le2}. This is impossible by (1) and (2).

  (4) Suppose that  $v_iv_j$ is a negative edge of  $\dot{G}$. If $b=4$, then $v_i$, $v_j$ have four common neighbours which denoted by $v_k$, $v_l$, $v_m$, $v_s$. And at most one of $\{v_k, v_l, v_m, v_s\}$ is negatively adjacent to $v_i$ and $v_j$. If all of $\{v_k, v_l, v_m, v_s\}$ are positively adjacent to $v_i$ and $v_j$, then the remaining neighbour $v_{t_1}$ of $v_i$ is negatively adjacent to $v_i$. So we have $v_{t_1}\overset{+}{\sim}v_k, v_l, v_m, v_s$ by the negative edge $v_i v_{t_1}$. For the positive edge  $v_iv_k$, we have $a=-4$ by and (1) and Lemma \ref{le2}. If one of $\{v_k, v_l, v_m, v_s\}$ is negatively adjacent to $v_i$ and $v_j$, we suppose $v_k\overset{-}{\sim}v_i,v_j$. For the negative edge $v_iv_k$, $v_k$ must be positively adjacent to  at least two of $\{v_l, v_m, v_s\}$. Without loss of generality, we suppose $v_k\overset{+}{\sim}v_l,v_m$. Then we also have $a=-4$ by the positive edge  $v_iv_l$.

\end{proof}

The following  is a classification discussion according to whether $\dot{G}$ contains triangles and whether the triangles are balanced.

\textbf{Firstly, we concern that $\dot{G}$  contains an unbalanced triangle.} There are two types of unbalanced triangles. Then first type of unbalanced triangles have only one negative edge and the second type of unbalanced triangles have three negative edges. We consider the first type now.

\begin{lemma}
  Let $\dot{G}\in \mathcal{C}_1\bigcup\mathcal{C}_4\bigcup\mathcal{C}_5$ be a connected non-complete 6-regular  and 2 net-regular SRSG. If it has an unbalanced triangle of the first type, then  $a\in \{-4,-2,-1,0\}$.
\end{lemma}

\begin{proof}
  Let $v_iv_jv_k$ be an unbalanced triangle of $\dot{G}$ such that $v_i\overset{+}{\sim}v_j$, $v_i\overset{-}{\sim}v_k$, $v_j\overset{+}{\sim}v_k$. Then there is another  common neighbour $v_l$ of $v_i,v_j$ such that $v_l\overset{+}{\sim}v_i$ and $v_l\overset{-}{\sim}v_j$ by Lemma \ref{le2}. So $v_i$ and $v_j$ have two negative walks of length 2. According to the vertex net-degree and result (1) of Lemma \ref{6-2-a}, we have $a\in \{-4,-2,-1,0,1\}$ by the positive edge $v_iv_j$.

  If $a=1$, then there are three positive and two negative walks of length 2 between $v_i$ and $v_j$. Let $N(v_i)\bigcap N(v_j)=\{v_k,v_l,v_m,v_s,v_t\}$ and suppose that $v_k\overset{-}{\sim}v_i$, $v_k\overset{+}{\sim}v_j$,  $v_l\overset{+}{\sim}v_i$, $v_l\overset{-}{\sim}v_j$,  $v_m\overset{-}{\sim}v_i, v_j$ and $v_s,v_t\overset{+}{\sim}v_i, v_j$.  We consider the positive edge $v_iv_l$. Since $d^{-}(v_m)=2$,   $v_l\overset{+}{\sim}v_m$. And so $v_l\overset{+}{\sim}v_s,v_t$ and  $v_k\overset{-}{\sim}v_l$.  We have $v_s$ is negatively adjacent to one of $v_k, v_m$ by the positive edge $v_iv_s$.  But this is contradictory to the negative degree of that vertex.

\end{proof}

\begin{figure}
  \centering
  \includegraphics[width=7cm]{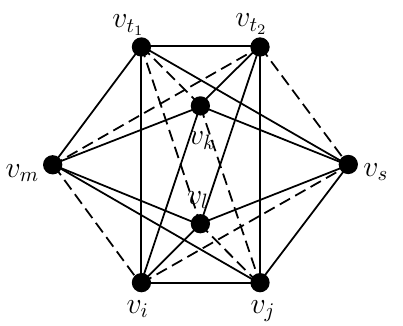}
  \caption{$\dot{S}^2_8$}\label{86-446}
\end{figure}

  If  there is an unbalanced triangle of the first type, then we have $a\in \{-4,-2,-1,0\}$ by the above Lemma and $-1\leq b\leq 4$ by Lemma \ref{le2} and Lemma \ref{6-2-a}. And by Lemma \ref{6-2-a}, we have $a=-4$ if $b=4$ and then $\dot{G}$ corresponds to the SRSG $\dot{S}^2_8$.  Then we need to consider the cases of $(a,b)=(0,3)$, $(0,2)$, $(0,1)$, $(0,0)$, $(0,-1)$, $(-1,3)$, $(-1,2)$, $(-1,1)$, $(-1,0)$, $(-1,-1)$, $(-2,3)$, $(-2,2)$, $(-2,1)$, $(-2,0)$, $(-2,-1)$. The case of $(-1,1)$ is impossible by Lemma \ref{6-2-a-1b1}.

We can rule out some cases of parameters $(a,b)$ by the following Lemmas.



\begin{lemma}
  Let $\dot{G}\in \mathcal{C}_1\bigcup\mathcal{C}_4\bigcup\mathcal{C}_5$ be a connected and non-complete 6-regular and 2 net-regular SRSG, then  $(a,b)\neq(-1,-1), (0,-1)$.
\end{lemma}

\begin{proof}
Let $v_iv_j$ be a negative edge of $\dot{G}$. If $b=-1$, then there are two negative walks and one positive walk of length 2 between $v_i$ and $v_j$. Suppose that $N(v_i)\bigcap N(v_j)=\{v_k,v_l,v_m\}$, $N(v_i)=\{v_j,v_k,v_l,v_m,v_s,v_t\}$ and $v_k\overset{+} {\sim}v_i, v_k\overset{-} {\sim}v_j$, $v_l\overset{+} {\sim}v_i, v_j$, $v_m\overset{-} {\sim}v_i, v_m\overset{+} {\sim}v_j$. Then $v_s, v_t\overset{+} {\sim}v_i$.

If $a=-1$,  it must have $v_s,v_t\overset{+} {\sim}v_m$ since $v_s,v_t\nsim v_j$ and $a=-1$.  But there are two positive walks of length 2 between $v_i$ and $v_m$, which is contradictory to $b=-1$.

If $a=0$, then there are two negative walks and two positive walks of length 2 or no walks of length 2 between two positively adjacent vertices. Vertices $v_i, v_k$ need two negative walks and one positive walk of length 2 by the positive edge  $v_iv_k$. Then it must have $v_k\overset{+} {\sim}v_m$ and $v_k$ must be adjacent to one of $\{v_s,v_t\}$. We suppose $v_k\sim v_s$ without loss of generality. Then $v_s$  must be  positively adjacent to $v_m$ by $a=0$. Now there are  two positive walks of length 2 between $v_i$ and $v_m$, which is contradictory to $b=-1$.

\end{proof}

\begin{lemma}
  Let $\dot{G}\in \mathcal{C}_1\bigcup\mathcal{C}_4\bigcup\mathcal{C}_5$ be a connected non-complete 6-regular  and 2 net-regular SRSG. If it has an unbalanced triangle of the first type,  then $(a,b)\neq(0,1), (0,2), (0,3)$.
\end{lemma}
\begin{proof}
Suppose that $v_iv_jv_k$ is an unbalanced triangle of first type in $\dot{G}$ such that $v_i\overset{+} {\sim}v_j$, $v_k\overset{+} {\sim}v_i$ and $v_k\overset{-} {\sim}v_j$. We firstly concern the positive edge $v_iv_j$. Since $a=0$, $v_i$ have the other three common neighbours $v_l,v_m,v_s$ with $v_j$. We suppose that $v_l\overset{-} {\sim}v_i$ and $v_l\overset{+} {\sim}v_j$. And there are two cases for $v_m,v_s$: One vertex is positively  adjacent to  $v_i,v_j$ and the other one is negatively  adjacent to  $v_i,v_j$; Both $v_m$ and $v_s$ are positively  adjacent to  $v_i,v_j$. We suppose $v_{t_1}$ and $v_{t_2}$ are the remaining neighbours of $v_i$ and $v_j$, respectively.

Case 1. $v_m\overset{+} {\sim}v_i,v_j$ and $v_s\overset{-} {\sim}v_i,v_j$. Then $v_{t_1}\overset{+} {\sim}v_i$ and $v_{t_2}\overset{+} {\sim}v_j$. We consider the negative edge $v_jv_s$, there is one negative  walk of length 2 between $v_j$ and $v_s$. Then  there must be the other one negative  walk of length 2 between them by Lemma \ref{le2}. So we have $b\leq 1$. If $b=1$, then $v_s\overset{+} {\sim}v_k, v_l, v_m, v_{t_2}$. Now there are two positive walks and two negative walks of length 2  between $v_i$ and $v_s$. Since $d(v_s)=6$,  we have $v_s{\nsim}v_{t_1}$.   Therefore, we get that $b=0$ by edge $v_iv_s$. A contradiction.
%

Case 2. $v_m,v_s\overset{+} {\sim}v_i,v_j$.  In this case, we have $v_{t_1}\overset{-} {\sim}v_i$ and $v_{t_2}\overset{-} {\sim}v_j$. And we have $v_k, v_l\overset{-} {\nsim}v_m,v_s$ by the positive edges $v_iv_k$ and $v_jv_l$. There are three cases for the positive edge $v_iv_m$ and we show them in Figure \ref{Three subcases}.



 If $b=1$, we consider the negative edge $v_jv_k$ for Subcase 2.1 and Subcase 2.3, then we have $v_k\overset{+} {\sim} v_{t_2}$. And we have $v_k\overset{-} {\sim} v_l$, $v_k\overset{+} {\sim} v_s$ or $v_k\overset{+} {\sim} v_l$, $v_k\overset{-} {\sim} v_s$. In either case, the positive edge $v_iv_k$ cannot satisfy  $a=0$ since  there are three positive walks of length 2 between $v_i$ and $v_k$ for the first case and three negative walks of length 2 for the second case.

 We consider the negative edge $v_iv_l$ for Subcase 2.2, then we have $v_l\overset{+} {\sim} v_{t_1}$. And we have $v_l\overset{-} {\sim} v_k$, $v_l\overset{+} {\sim} v_s$ or $v_l\overset{+} {\sim} v_k$, $v_l\overset{-} {\sim} v_s$. In either case, the positive edge $v_jv_l$ cannot satisfy  $a=0$ since  there are  three positive walks of length 2 between $v_j$ and $v_l$ for the first case and three negative walks of length 2 for the second case.

If $b=2$, then there are only two positive walks of length 2 between two negatively adjacent vertices.

For Subcase 2.1: In this case, the negative edges $v_iv_l$ and $v_jv_k$ satisfy  $b=2$ and  the positive edges $v_jv_m$ and $v_iv_m$ satisfy  $a=0$. Then $v_l\nsim v_k, v_s, v_{t_1}$ and $v_k\nsim v_{t_2}, v_s$ and $v_m\nsim v_{t_1}, v_{t_2}$.
 Now we consider the positive edge $v_iv_k$, it cannot satisfy  $a=0$ by Lemma \ref{le2} since $v_k\nsim v_{t_2}, v_l, v_s$.

For Subcase 2.2: This case contradicts $b=2$, since there is one negative walk of length 2 between $v_i$ and $v_{t_1}$.

For Subcase 2.3: In this case, the negative edge $v_jv_k$ satisfies  $b=2$, then $v_k\nsim v_l, v_s, v_{t_2}$. Therefore, $v_k$ have at most three common neighbours with $v_i$. Then $v_iv_k$ cannot satisfy  $a=0$.

If $b=3$:

For Subcase 2.1: The positive edges $v_jv_m$ and $v_iv_m$ satisfy  $a=0$ in this case. So $v_m\nsim v_{t_1}, v_{t_2}$. Since $b=3$, we have $v_{t_1}\overset{+} {\sim}v_k,v_s$, $v_{t_1}\overset{-} {\sim}v_l$ and $v_{t_2}\overset{+} {\sim} v_l,v_s$, $v_{t_2}\overset{-} {\sim} v_k$. Now the negative edges $v_jv_k$ and $v_iv_l$ satisfy  $b=3$, so $v_l\nsim v_k,v_s$ and $v_k\nsim v_s$. Then there are only three common neighbours for  $v_i$ and $v_s$, which is contradictory to $a=0$.

%

For Subcase 2.2:  This case contradicts $b=3$, since there is one negative walk of length 2 between $v_i$ and $v_{t_1}$.
%
%

For Subcase 2.3: We have $v_k\overset{-} {\nsim}v_{t_1}$ and $v_k\overset{-} {\nsim}v_{l}$ by negative edges $v_iv_{t_1}$ and $v_jv_{k}$, respectively.  We consider the positive edge $v_iv_k$, then $v_k\overset{+} {\sim}v_s,v_l$ or $v_k\overset{+} {\sim}v_s,v_{t_1}$ to  satisfy  $a=0$. If $v_k\overset{+} {\sim}v_s,v_l$, then there are four positive walks of length 2 between $v_j$ and $v_k$. This contradicts  $b=3$. If $v_k\overset{+} {\sim}v_s,v_{t_1}$, the negative edge $v_jv_k$ satisfies  $b=3$. So $v_k\nsim v_l, v_{t_2}$. We  consider the negative edge $v_jv_{t_2}$, then $v_{t_2}\overset{+} {\sim}v_m,v_l, v_s$. But there will be three negative walks of length 2 between $v_j$ and $v_m$, which is contradictory to $a=0$.

\end{proof}

\begin{figure}
  \centering
    \subfigure[Subcase 2.1]{
  \includegraphics[width=5CM]{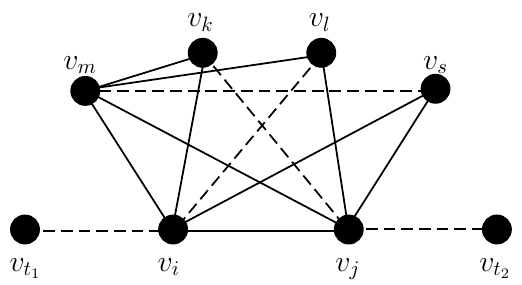}}
    \subfigure[Subcase 2.2]{
  \includegraphics[width=5CM]{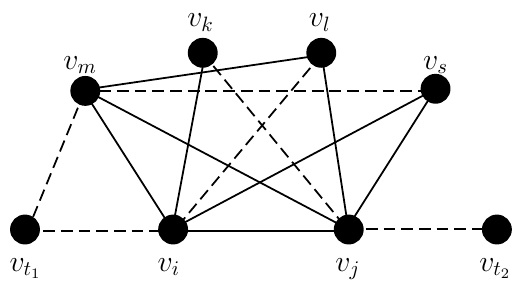}}
      \subfigure[Subcase 2.3]{
  \includegraphics[width=5CM]{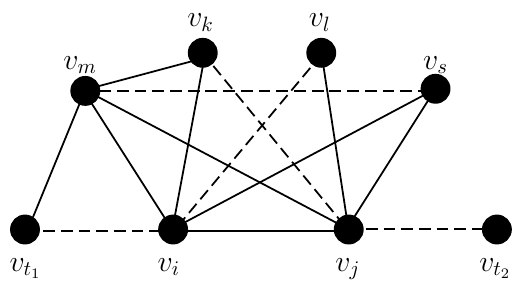}}
  \centering
  \caption{Three subcases of Case 2 for $a=0$}\label{Three subcases}
\end{figure}

\begin{lemma}
  Let $\dot{G}\in \mathcal{C}_1\bigcup\mathcal{C}_4\bigcup\mathcal{C}_5$ be a connected and non-complete 6-regular and 2 net-regular SRSG. If it has an unbalanced triangle of the first type,  then $(a,b)\neq(-2,-1), (-2,0), (-2,1)$.
\end{lemma}
\begin{proof}
Suppose $v_iv_jv_k$ is an unbalanced triangle of the first type in $\dot{G}$ such that $v_i\overset{-} {\sim}v_j$, $v_i\overset{+} {\sim}v_k$ and $v_j\overset{+} {\sim}v_k$. If $a=-2$, then there are only two negative walks of length 2 between any two positively adjacent vertices. We consider the negative edge $v_iv_j$.

If $b=-1$ or 0, then  $v_i$ and $v_j$ need  two negative walks of length 2. So we suppose $v_l,v_m\in N(v_i)\bigcap N(v_j)$ and $v_l\overset{+} {\sim}v_i$, $v_l\overset{-} {\sim}v_j$ and  $v_m\overset{-} {\sim}v_i$, $v_m\overset{+} {\sim}v_j$. But there is  one positive walk of length 2 between $v_i$ and $v_l$, which is contradictory to  $a=-2$.

If $(a,b)=(-2,1)$, then $v_i$ and $v_j$ must have only one common neighbour, denoted by $v_k$, by the above argument.  We consider the positive edge $v_iv_k$, then $v_i$ and $v_k$ need another negative walk of length 2. So we suppose  $v_l\in N(v_i)\bigcap N(v_k)$ and $v_l\overset{+} {\sim}v_i$, $v_l\overset{-} {\sim}v_k$. Now $v_i$ and $v_l$ also need another negative walk of length 2. Let $v_m$ be the negative neighbour of $v_i$. So  $v_l\overset{+} {\sim}v_m$ since  $v_l\nsim v_j$. Since positive edge $v_iv_k$ has satisfied $a=-2$,  $v_k\nsim v_m$ and so there must be only one positive walk of length 2 between $v_i$ and $v_m$.  Now the edges $v_iv_j$, $v_iv_k$, $v_iv_l$ and $v_iv_m$ satisfy that $(a,b)=(-2,-1)$. So the remaining two positive neighbours of $v_i$ have at most one common neighbour with $v_i$. This is contradictory to  $a=-2$.

\end{proof}

\begin{lemma}\label{6-b=3}
  Let $\dot{G}\in \mathcal{C}_1\bigcup\mathcal{C}_4\bigcup\mathcal{C}_5$ be a connected and non-complete 6-regular and 2 net-regular SRSG. If it has an unbalanced triangle of the first type and $b=3$, then $\dot{G}$ must have an unbalanced triangle of the second type.
\end{lemma}
\begin{proof}
  Let $v_iv_jv_k$ be an unbalanced triangle of $\dot{G}$ such that $v_i\overset{-}{\sim}v_j$, $v_i\overset{+}{\sim}v_k$, $v_j\overset{+}{\sim}v_k$. If $b=3$, then there are only three positive walks of length 2 between $v_i$  and $v_j$. Suppose  $N(v_i)=\{v_j,v_k,v_l,v_m,v_{s_1},v_{t_1}\}$ and $N(v_j)=\{v_i,v_k,v_l,v_m,v_{s_2},v_{t_2}\}$.   If $v_l\overset{-}{\sim}v_i,v_j$ or  $v_m\overset{-}{\sim}v_i,v_j$, the result is true obviously. Otherwise,  $v_l,v_m\overset{+}{\sim}v_i$.  Then we suppose that $v_{s_1}\overset{-}{\sim}v_i$ and $v_{t_1}\overset{+}{\sim}v_i$. For the negative edge $v_{s_1}v_i$, at least one of  $\{v_k,v_l,v_m\}$ is positively adjacent to $v_{s_1}$. Then we have $a=-4$ or $-3$ by Lemma \ref{le2}. This is impossible by Lemma \ref{6-2-a}. So it must have that one of $v_l,v_m$ is negatively adjacent to $v_i$.
\end{proof}

\begin{lemma}\label{6-(a,b)=(-2,3)(-1,3)}
  Let $\dot{G}\in \mathcal{C}_1\bigcup\mathcal{C}_4\bigcup\mathcal{C}_5$ be a connected and non-complete 6-regular and 2 net-regular SRSG. If it has an unbalanced triangle of the first type,  then $(a,b)\neq(-2,3)$. And if $(a,b)=(-1,3)$, $\dot{G}$ is isomorphic to $\dot{S}_9$ as shown in Figure \ref{S9}.
\end{lemma}
\begin{proof}
Let $v_iv_jv_k$ be an unbalanced triangle of $\dot{G}$ such that $v_i\overset{-}{\sim}v_j$, $v_i\overset{+}{\sim}v_k$, $v_j\overset{+}{\sim}v_k$. Suppose  $N(v_i)=\{v_j,v_k,v_l,v_m,v_{s_1},v_{t_1}\}$ and $N(v_j)=\{v_i,v_k,v_l,v_m,v_{s_2},v_{t_2}\}$. By the Lemma \ref{6-b=3}, one of $v_l,v_m$ is negatively adjacent to $v_i$. So we suppose $v_l\overset{-}{\sim}v_i,v_j$ and $v_m\overset{+}{\sim}v_i,v_j$. If $a=-2$ or $a=-1$,  there must be two negative walks of length 2 between two positively adjacent vertices. Then we have $v_{s_1},v_{t_1}, v_{s_2},v_{t_2}\overset{+}{\sim}v_l$ by the positive edges $v_iv_{s_1}$, $v_iv_{t_1}$, $v_jv_{s_2}$ and $v_jv_{t_2}$. Then we can get that $c\leq-2$ by vertices $v_{s_1}$ and $v_j$. By the equation \eqref{6-2-abcn}, we have  $c=0$ if $(a,b)=(-2,3)$. A contradiction.

Now we consider $(a,b)=(-1,3)$. The possible parameters sets are $(11,6,-1,3,-1)$, $(9,6,-1,3,-2)$ and $(8,6,-1,3,-4)$ by equation \eqref{6-2-abcn}. Since $c\leq-2$, we concern  $(9,6,-1,3,-2)$ and $(8,6,-1,3,-4)$. If  $\dot{G}$ has parameters $(8,6,-1,3,-4)$,  two adjacent vertices in  underlying graph $G$ have three common neighbours and  two non-adjacent vertices have four common neighbours. Then $G$ is a SRG with parameters $(8,6,3,4)$. But such SRG does not exist by equation \eqref{srg}.  If  $\dot{G}$ has parameters $(9,6,-1,3,-2)$,  two adjacent vertices in  $G$ have three common neighbours and  two non-adjacent vertices have two or six common neighbours. We can get that there is only one 6-regular graph with order 9 satisfies these conditions, i.e. $K_{3,3,3}$, by checking all the 6-regular graphs with order 9 (Figure \ref{9-6}). Suppose $v_1,v_2,v_3$ are three vertices of $K_{3,3,3}$ and they are not adjacent to each other. Since $c=-2$, there are two positive and four negative walks of length 2 between $v_1$ and $v_2$. So we suppose  $v_4,v_5\overset{+}{\sim}v_1$,  $v_4,v_5\overset{-}{\sim}v_2$,  $v_6,v_7\overset{-}{\sim}v_1$,  $v_6,v_7\overset{+}{\sim}v_2$ and  $v_8,v_9\overset{+}{\sim}v_1, v_2$. Since $v_3\nsim v_1,v_2$, we have $v_3\overset{+}{\sim}v_6,v_7$ and  $v_3\overset{+}{\sim}v_4,v_5$ and  $v_3\overset{-}{\sim}v_8,v_9$. Now, there are three positive walks of length 2 between $v_4$ and $v_5$. So do vertices $v_6$ and $v_7$, $v_8$ and $v_9$. Since $(a,b,c)= (-1,3,-2)$, there are one positive and two negative walks of length 2 between two positively adjacent vertices and three positive walks of length 2 between two negatively adjacent vertices. Therefore, $v_4\overset{-} {\sim}v_5$, $v_6\overset{-}{\sim}v_7$ and $v_8\overset{-}{\sim}v_9$. Concern the positive edge $v_1v_4$, vertices $v_1$ and $v_4$ need one positive walk and one negative walk of length 2. Then, we have $v_6\overset{+} {\sim}v_4$ or $v_7\overset{+}{\sim}v_4$, and $v_8\overset{+}{\sim}v_4$ or $v_9\overset{+}{\sim}v_4$. We discuss the case of $v_6\overset{+}{\sim}v_4$ and $v_8\overset{+}{\sim}v_4$ in details. Since $v_4{\sim}v_6,v_8$, we have $v_4{\sim}v_7,v_9$. By $c=-2$, we have $v_5,v_8\overset{+} {\sim}v_7$ and $v_5,v_6\overset{+} {\sim}v_9$. Therefore, $\dot{G}$  is isomorphic to $\dot{S}_9$. By the same way, we also  get the same result for the remaining three cases.
\end{proof}

\begin{figure}
  \centering
  \includegraphics[width=7cm]{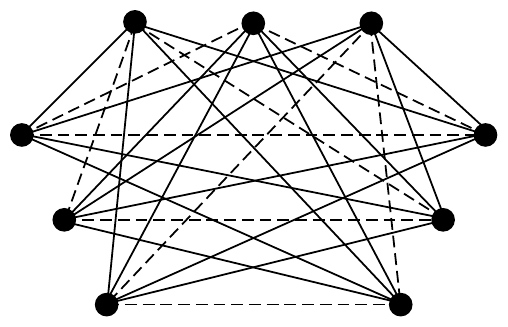}
  \caption{$\dot{S}_9$ with $(9,6,-1,3,-2)$}\label{S9}
\end{figure}

\begin{remark}
The data of the regular graphs is obtained from the following website and \cite{M}:

http://www.mathe2.uni-bayreuth.de/markus/reggraphs.html\#CRG
\end{remark}

 Let $\dot{G}\in \mathcal{C}_1\bigcup\mathcal{C}_4\bigcup\mathcal{C}_5$ be a connected non-complete 6-regular and 2 net-regular SRSG.  For the cases of $(a,b)=(0,0)$, $(-1,0)$,$(-1,2)$ and $(-2,2)$,  we have the corresponding parameters sets are $(9,6,0,0,-1)$, $(8,6,0,0,-2)$, $(9,6,-1,0,1)$, $(8,6,-1,0,2)$, $(9,6,-1,2,-1)$, $(8,6,-1,2,-2)$,  $(9,6,-2,2,1)$, $(8,6,-2,2,2)$ by the equation \eqref{6-2-abcn}.

If $\dot{G}$ has parameters $(9,6,0,0,-1)$ or $(9,6,-1,2,-1)$,  $c=-1$ implies that there are one positive and two negative walks of length 2 between any two non-adjacent vertices in $\dot{G}$ by Lemma \ref{le2}. Then any two non-adjacent vertices in underlying graph ${G}$ have three common neighbours. There are four 6-regular graphs with order 9 (Figure \ref{9-6}). It is easy to check that none of them satisfies this condition.

\begin{lemma}
 Let $\dot{G}\in \mathcal{C}_1\bigcup\mathcal{C}_4\bigcup\mathcal{C}_5$ be a connected and non-complete 6-regular and 2 net-regular SRSG.  If $\dot{G}$ has parameters $(9,6,-1,0,1)$, then $\dot{G}=\dot{S}^1_9$ (see in Figure \ref{(9,6,-1,0,1)}).
\end{lemma}
\begin{proof}
  If $\dot{G}$ has parameters $(9,6,-1,0,1)$, $c=1$ implies that two non-adjacent vertices in $G$ have one or five common neighbours and $(a,b)=(-1,0)$ implies that two adjacent vertices in $G$ have three or four common neighbours by Lemma \ref{le2}. Only the graph $G_9$  in Figure \ref{(9,6,-1,0,1)} satisfies these conditions by simple verification for all 6-regular graphs with order 9. If two adjacent vertices in  $G_9$ have three (four, respectively) common neighbours, then the corresponding edge in $\dot{G}$ is positive (negative, respectively). So we get two signed graphs $\dot{S}^1_9$ in  Figure \ref{(9,6,-1,0,1)}.
\end{proof}

\begin{figure}
  \centering
  \subfigure[$G_9$]{\includegraphics[width=5CM]{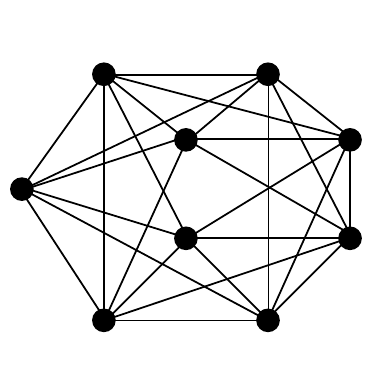}}
   \subfigure[$\dot{S}^1_9$]{\includegraphics[width=5.2cm]{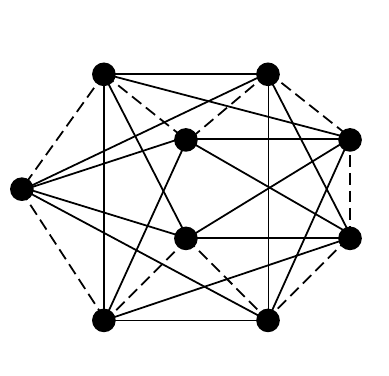}}
  \caption{ $G_9$ and $\dot{S}^1_9$ }\label{(9,6,-1,0,1)}
\end{figure}

\begin{lemma}
 Let $\dot{G}\in \mathcal{C}_1\bigcup\mathcal{C}_4\bigcup\mathcal{C}_5$ be a connected non-complete 6-regular and 2 net-regular SRSG.  If $\dot{G}$ has parameters $(8,6,0,0,-2)$, then $\dot{G}$ is isomorphic to  $\dot{S}^3_8$ (see in Figure \ref{8600-2}).
\end{lemma}
\begin{proof}
  Suppose $\dot{G}$ has parameters $(8,6,0,0,-2)$.  There is only one 6-regular graph with order 8 ($G_8$ in Figure \ref{8-6.864-4-6}), in which any two adjacent vertices have four common neighbours and any two non-adjacent vertices have six common neighbours. So there are  two positive and two negative walks of length 2 between any two adjacent vertices and two positive and four negative walks of length 2 between any two non-adjacent vertices in $\dot{G}$. Let $v_1v_2$ be a negative edge of $\dot{G}$ and $N(v_1)\bigcap N(v_2)=\{v_3,v_4,v_5,v_6\}$. Suppose that $v_3\overset{+}{\sim}v_1$, $v_3\overset{-}{\sim}v_2$, $v_4\overset{-}{\sim}v_1$, $v_4\overset{+}{\sim}v_2$ and $v_5,v_6\overset{+}{\sim}v_1,v_2$ and $v_7\overset{+}{\sim}v_1$,  $v_8\overset{+}{\sim}v_2$. For the positive edge $v_1v_3$, it must have $v_3\overset{+}{\sim}v_4$. We concern two non-adjacent vertices $v_2$ and $v_7$. Then it must have $v_3\overset{+}{\sim}v_7$ and two of $\{v_4,v_5,v_6,v_8\}$ are positively adjacent to $v_7$ and the remaining two vertices are negatively adjacent to $v_7$. Then we have six cases: $v_4,v_5\overset{-}{\sim}v_7$, $v_6,v_8\overset{+}{\sim}v_7$ or $v_4,v_6\overset{-}{\sim}v_7$, $v_5,v_8\overset{+}{\sim}v_7$ or $v_4,v_8\overset{-}{\sim}v_7$, $v_5,v_6\overset{+}{\sim}v_7$ or $v_5,v_6\overset{-}{\sim}v_7$, $v_4,v_8\overset{+}{\sim}v_7$ or $v_5,v_8\overset{-}{\sim}v_7$, $v_4,v_6\overset{+}{\sim}v_7$ or $v_6,v_8\overset{-}{\sim}v_7$, $v_4,v_5\overset{+}{\sim}v_7$. There are three positive walks, four positive walks, three negative walks of length 2 between $v_1$ and $v_7$ for the first and second cases, the third case, the fourth case, respectively. They contradict $a=0$. So we only show the Cases 5,6 in Figure \ref{8600-2(0)}.

  For these two cases, we consider two non-adjacent vertices $v_1$ and $v_8$. Then it must have $v_4\overset{+}{\sim}v_8$. For the other one negative walk of length 2, these two cases have three possibilities ($v_3\overset{-}{\sim}v_8$ or $v_5\overset{-}{\sim}v_8$ or $v_6\overset{-}{\sim}v_8$), respectively. Then we also have six subcases. But if $v_3\overset{-}{\sim}v_8$, then there are four positive walks of length 2 for Cases 5 and 6. If  $v_5\overset{-}{\sim}v_8$ for Case 5 and $v_6\overset{-}{\sim}v_8$ for Case 6, there are four positive walks of length 2 between  $v_7$ and $v_8$. So it remains two feasible subcases as shown in  Figure \ref{8600-2(0)}. We have $v_3\overset{-}{\sim}v_5$ or $v_3\overset{-}{\sim}v_6$ for these two subcases since the vertex $v_3$ need one negative neighbour. But in fact it must be $v_3\overset{-}{\sim}v_6$ for the first subcase. Otherwise, $v_3\overset{-}{\sim}v_5$ will lead to that there are three negative walks of length 2 between $v_1$ and $v_5$, which is contradictory to $a=0$. Therefore, we also have $v_4\overset{-}{\sim}v_5$ and $v_5\overset{+}{\sim}v_6$ for the first subcase by the net-degree of vertices $v_4$ and $v_5$. In a similar way, we get that $v_3\overset{-}{\sim}v_5$, $v_4\overset{-}{\sim}v_6$ and $v_5\overset{+}{\sim}v_6$ for the second subcase. So we obtain two SRSGs which shown in Figure \ref{8600-2-4}. It is easy to see that $\dot{G}^1_8$ is isomorphic to $\dot{G}^2_8$  (The isomorphic mapping is $f$: $V(\dot{G}^1_8)\rightarrow\dot{G}^2_8$ satisfies that $f(v^1_i)=v^2_i$, $i\neq5,6$, $f(v^1_5)=v^2_6$, $f(v^1_6)=v^2_5$). We choose the signed graph $\dot{G}^1_8$ and denote it by $\dot{S}^3_8$ (Figure \ref{8600-2}).

\end{proof}

\begin{figure}
  \centering
  \subfigure[The Cases 5,6 for vertices $v_2$ and $v_7$]{
  \includegraphics[width=8cm]{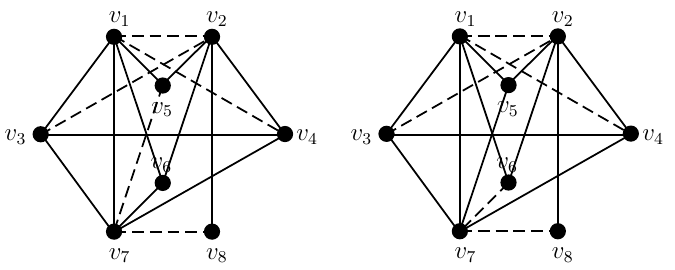}}
  \subfigure[The two feasible cases for vertices $v_1$ and $v_8$]{
  \includegraphics[width=8cm]{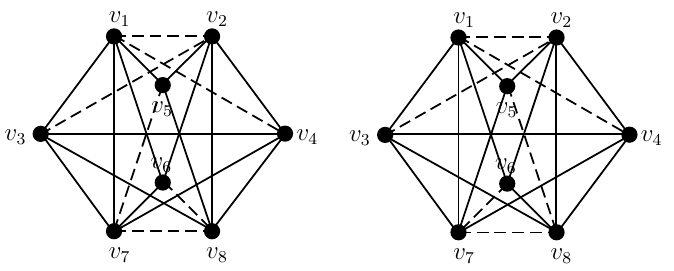}}
  \caption{About $(8,6,0,0,-2)$}\label{8600-2(0)}
\end{figure}
\begin{figure}
  \centering
  \subfigure[$\dot{G}^1_8$]{
  \includegraphics[width=3.7cm]{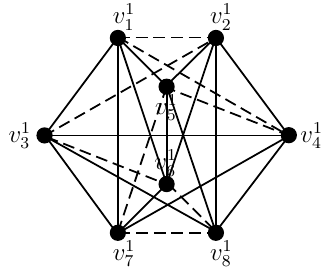}}
    \subfigure[$\dot{G}^2_8$]{
  \includegraphics[width=3.7cm]{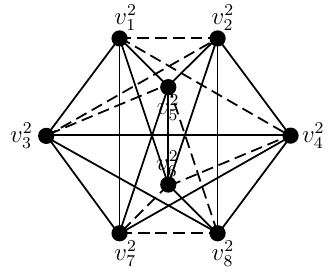}}
  \caption{SRSGs $\dot{G}^1_8$ and $\dot{G}^2_8$}\label{8600-2-4}
\end{figure}

\begin{figure}
  \centering
  \includegraphics[width=5cm]{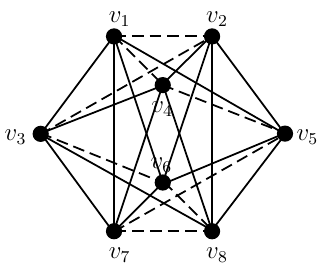}
  \caption{ SRSG $\dot{S}^3_8$ with $(8,6,0,0,-2)$}\label{8600-2}
\end{figure}

If $\dot{G}$ has parameters $(8,6,-1,0,2)$ or $(8,6,-1,2,-2)$, there exist two adjacent vertices which have three common neighbours in underlying graph since $a=-1$. But $G_8$ cannot satisfy the condition.

If $\dot{G}$ has parameters $(9,6,-2,2,1)$ or $(8,6,-2,2,2)$,  any two adjacent vertices in the underlying graph have two common neighbours by $(a,b)=(-2,2)$.  All 6-regular graphs with order 8 and 9 cannot satisfy this condition by some simple verification.

 \textbf{Now, we consider that $\dot{G}$ contains an unbalanced triangle of the second type.}

\begin{lemma}
  Let $\dot{G}\in \mathcal{C}_1\bigcup\mathcal{C}_4\bigcup\mathcal{C}_5$ be a connected and non-complete 6-regular and 2 net-regular SRSG. If $\dot{G}$ has an unbalanced triangle of the second type, then $b=1$ with $0\leq a\leq 3$ or  $b=3$ with $a=-1,-2$.
\end{lemma}

\begin{proof}
  Suppose that $v_iv_jv_k$ is  the second type of unbalanced triangle of $\dot{G}$. Then $v_i\overset{-} {\sim}v_j$, $v_i\overset{-} {\sim}v_k$ and $v_k\overset{-} {\sim}v_j$.

 We consider the negative edge $v_iv_j$. Both $v_i$ and  $v_j$ have four positive neighbours. Then it is obvious that $1\leq b\leq5$.

  If $b=4$ or $5$, then there is at least one common positive neighbour of $v_i$ and  $v_j$, denoted by $v_l$,  is positively adjacent to $v_k$. Since $v_l\overset{+} {\sim}v_i,v_j$, there are two negative walks of length 2 between $v_l$ and $v_i$. By the situation of these two walks and Lemma \ref{le2}, $v_l$ and $v_i$ must have the other  two negative walks of length 2. So  $a=-4$ or $-3$.  But we have $a\neq-3$ and the SRSG with $a=-4$ contains no unbalanced triangle of second type by Lemma \ref{6-2-a}.

  If $b=3$, suppose $N(v_i)\bigcap N(v_j)=\{ v_k,v_l, v_m\}$ and $v_{s_1},v_{t_1}\in N(v_i)$,  $v_{s_2},v_{t_2}\in  N(v_j)$. Then $v_l, v_m\overset{+} {\sim}v_i,v_j$ and $v_{s_1},v_{t_1}\overset{+} {\sim}v_i$ and $v_{s_2},v_{t_2}\overset{+} {\sim}v_j$. To avoid the cases of  $a=-4$ and $-3$, we have $v_k\nsim v_l, v_m$. Then it must have $v_k\overset{+} {\sim}v_{s_1},v_{t_1}, v_{s_2},v_{t_2}$ by $b=3$. So there must be two negative walks of length 2 between two positively adjacent vertices in this case. Therefore, we have $-2\leq a\leq1$ since $a\neq-4,-3$. If $a=0$, then both $v_{s_1}$ and $v_{t_1}$ have four common neighbours with $v_i$ and both $v_{s_2}$ and $v_{t_2}$ have four common neighbours with $v_j$. Therefore,  $v_{s_1},v_{t_1}, v_{s_2},v_{t_2}\sim v_l, v_m$ and $v_{s_1}\sim v_{t_1}$ and $v_{s_2}\sim v_{t_2}$. We consider the positive edge $v_iv_l$, then $v_l$ have four common neighbours with $v_i$. Then  $v_l\sim v_m$ since  $v_k\nsim v_l, v_m$ by $b=3$. But $d(v_l)=d(v_m)=7$ now, a contradiction.  If $a=1$, then
   both $v_{s_1}$ and $v_{t_1}$ have five common neighbours with $v_i$. But this is impossible since they are not adjacent to  $v_j$.

  If $b=2$, we suppose the remaining common neighbour of $v_i$ and  $v_j$ is $v_l$ and $v_m,v_s,v_t$ are the remaining positive neighbours of $v_i$. For the positive edge $v_iv_l$, there is one negative walk of length 2 between $v_i$ and $v_l$. Then they must have another one by Lemma \ref{le2}. So we have $a\leq 1$ but $a\neq-4,-3$ by Lemma \ref{6-2-a}. If $a=-2$ or $-1$, then there must be two negative walks of length 2 between any two positively adjacent vertices. Then $v_m,v_s,v_t\overset{+} {\sim}v_k$ since  $v_m,v_s,v_t\nsim v_j$. This is  contradictory to $b=2$ by negative edge $v_iv_k$.  If $a=0$, then $v_i$ and $v_l$ must have two negative and two positive walks of length 2. We could suppose that $v_l \overset{-} {\sim}v_m$ and $v_l\overset{+} {\sim} v_s,v_t$ if $v_l\nsim v_k$.  Then $v_i$  needs two negative and two positive walks of length 2 between $v_i$ and its every positive neighbour. So $v_m,v_s,v_t\overset{+} {\sim}v_k$. But this is  contradictory to $b=2$. If  $v_l\sim v_k$, then it must have $v_l\overset{+} {\sim} v_k$ since $d^-(v_k)=2$. But it will have $a\leq-3$ by positive edge $v_iv_l$. A contradiction.   If $a=1$, then there must be two negative and three positive walks of length 2 between $v_i$ and $v_l$. So $v_l$ is negatively adjacent to  $v_k$ by Lemma \ref{le2}.  Then we have  $d^-(v_k)=3$. A contradiction.

  If $b=1$, then $v_k$ is not adjacent to the remaining positive neighbours of $v_i$. So we have  $0\leq a\leq 3$.

\end{proof}

The possible values of $(a,b)$ we need to concern are $(-2,3)$, $(-1,3)$, $(0,1)$, $(1,1)$, $(2,1)$ and $(3,1)$ by the above Lemma.

If $(a,b)=(-2,3)$ or $(-1,3)$, we can get that   $\dot{G}$ must contain the first type of unbalanced triangles from the proof of above Lemma. So it  is isomorphic to $\dot{S}_9$ with parameters $(9,6,-1,3,-2)$ by the Lemmas \ref{6-b=3} and \ref{6-(a,b)=(-2,3)(-1,3)}. Now we consider the remaining values of $(a,b)$.

\begin{lemma}\label{6-2-b=1}
  Let $\dot{G}\in \mathcal{C}_1\bigcup\mathcal{C}_4\bigcup\mathcal{C}_5$ be a connected and non-complete 6-regular and 2 net-regular SRSG. If $\dot{G}$ has an unbalanced triangle of the second type and $b=1$, then $n\geq 15$.
\end{lemma}
\begin{proof}
   Suppose that $v_iv_jv_k$ is the unbalanced triangle of the second type of $\dot{G}$. Then $v_i\overset{-} {\sim}v_j$, $v_i\overset{-} {\sim}v_k$ and $v_k\overset{-} {\sim}v_j$.
   We consider the negative edge $v_iv_j$. Since  $b=1$, the positive neighbours of $v_i$ and $v_j$ are not adjacent to $v_k$. And the positive neighbours of $v_k$ are also not adjacent to $v_i$ and $v_j$. Then $n\geq 15$.
\end{proof}

If $(a,b)=(0,1)$, we have $c(n-7)=-4$ by equation \eqref{6-2-abcn}. Then $n=11,9,8$. So there is no SRSG with $(a,b)=(0,1)$ by the Lemma \ref{6-2-b=1}.

\begin{lemma}
  Let $\dot{G}\in \mathcal{C}_1\bigcup\mathcal{C}_4\bigcup\mathcal{C}_5$ be a connected and non-complete 6-regular and 2 net-regular SRSG. If $(a,b)=(1,1)$, then $\dot{G}$ is isomorphic to $\dot{S}_{15}$ (Figure \ref{(15,6,1,1,-1)}) and has parameters $(15,6,1,1,-1)$.
\end{lemma}

\begin{proof}
  If $(a,b)=(1,1)$, then there are two cases for the walks of length 2 between two adjacent vertices: The first is there is only one positive walk of length 2; The second is there are three positive and two negative walks of length 2. Let $v_iv_j$ be an edge of $\dot{G}$ and there are three positive and two negative walks of length 2 between $v_i$ and $v_j$. We suppose $N(v_i)\bigcap N(v_j)=\{v_k,v_l,v_m,v_s,v_t\}$. If edge $v_iv_j$ is negative, then we suppose $v_k,v_l,v_m\overset{+} {\sim}v_i,v_j$ and $v_s\overset{+} {\sim}v_i,v_s\overset{-} {\sim}v_j$ and $v_t\overset{-} {\sim}v_i,v_t\overset{+} {\sim}v_j$. Now we consider the negative edge $v_iv_t$. Since there is one negative walk of length 2 between $v_i$ and $v_t$,  there must be another  negative walk and three positive walks of length 2. Therefore, at least one of $v_k,v_l,v_m,v_s$ is positively adjacent to $v_t$. We suppose $v_l\overset{+} {\sim}v_t$ without loss of generality. Then there are two  negative walks of length 2 between $v_i$ and $v_l$. Since these two common neighbours are joined to $v_i$ by negative edges and joined to $v_l$ by positive edges,  there are another two negative walks of length 2  between $v_i$ and $v_l$ by Lemma \ref{le2}. So we have $a\leq -3$, which contradicts $a=1$. Then there is only  one positive walk of length 2 between two negatively adjacent vertices in $\dot{G}$.

  If $v_iv_j$ is positive, we suppose $v_k,v_l\overset{+} {\sim}v_i,v_j$,\; $v_m\overset{-} {\sim}v_i,v_j$,\; $v_s\overset{+} {\sim}v_i,v_s\overset{-} {\sim}v_j$ and $v_t\overset{-} {\sim}v_i,v_t\overset{+} {\sim}v_j$. Now we consider the positive edge $v_iv_s$. There must be another one negative walk and three positive walks of length 2 between $v_i$ and $v_s$. Therefore, $v_k,v_l,v_m,v_t\sim v_k$. For vertex $v_m$,  there are two negative walks of length 2 between $v_m$ and $v_j$ if $v_m\overset{+} {\sim}v_s$. This is contradictory to $b=1$. If $v_m\overset{-} {\sim}v_s$, then there are two negative walks of length 2 between $v_m$ and $v_i$. According to the situation of these two walks, we get that $b\leq -3$ by Lemma \ref{le2}. But this also contradicts $b=1$. Therefore, there is only  one positive walk of length 2 between any two positively adjacent vertices in $\dot{G}$.

   To sum up, there is only  one positive walk of length 2 between any two adjacent vertices in $\dot{G}$ and the corresponding triangle must be $+K_3$ or $-K_3$.

   By equation \eqref{6-2-abcn} and Lemma \ref{6-2-b=1}, we only need to consider the parameters $(15,6,1,1,-1)$.  $c=-1$ implies that there is only one positive walk and  two negative walks of length 2 between any two non-adjacent vertices in $\dot{G}$ by Lemma \ref{le2} and $r=6$. Since there is only  one positive walk of length 2 between any two adjacent vertices, the underlying graph of $\dot{G}$ is a strongly regular graph with parameters $(15,6,1,3)$, i.e. $GQ(2,2)$ (Figure \ref{(15,6,1,1,-1)}). We suppose
   $v_1v_2v_{12}$ is a $-K_3$ in $\dot{G}$. Since $d^- (v_1)=d^-(v_2)=d^-(v_{12})=2$,  edges $v_1v_7$, $v_1v_9$, $v_1v_ {10}$, $v_1v_{15}$, $v_2v_3$, $v_2v_5$, $v_2v_6$, $v_2v_{14}$, $v_{12}v_4$,  $v_{12}v_8$, $v_{12}v_{11}$, $v_{12}v_{13}$ are positive. Therefore, edges $v_3v_{14}$, $v_4v_{13}$, $v_5v_6$,  $v_7v_{15}$, $v_8v_{11}$, $v_9v_{10}$ are also positive.  We consider two non-adjacent vertices $v_1$ and $v_3$. Vertices $v_7$ and $v_{10}$ are their common neighbours. Without loss of generality, we suppose $v_{10}\overset{-} {\sim}v_3$ and $v_7\overset{+} {\sim}v_3$. Therefore, edges $v_4v_{10}$ and $v_3v_4$ are negative and edges $v_3v_{8}$ and $v_7v_8$ are positive. By the vertex ne-degree, we get that edges $v_7v_6$, $v_7v_ {13}$, $v_8v_5$ and $v_8v_9$ are negative. And according to $b=1$, we have edges $v_6v_{13}$ and $v_5v_9$ are negative. Then by the net-degree of vertices $v_5$ and $v_6$, we obtain that edges $v_5v_4$, $v_5v_{15}$, $v_6v_{10}$ and $v_6v_{11}$ are positive. Therefore, edges $v_4v_ {15}$ and $v_{10}v_{11}$ are positive. By vertex  $v_{15}$, we get that triangle $v_ {11}v_{14}v_{15}$ is a $-K_3$. And the edges of triangle $v_9v_ {13}v_{14}$ are positive by the net-degree of $v_{14}$. Then we get the signed graph $\dot{S}_{15}$ in Figure \ref{(15,6,1,1,-1)}. If $v_{10}\overset{+} {\sim}v_3$ and $v_7\overset{-} {\sim}v_3$, then we get the signed graph $\dot{G}^1_{15}$ which is isomorphic to $\dot{S}_{15}$. The isomorphic mapping is $f_1:V(\dot{S} _{15})\rightarrow V(\dot{G}^1_{15})$ which satisfies that
\begin{align*}
   & f_1(v_1(\dot{S} _{15}))=v_{11}(\dot{G}^1_{15}),f_1(v_2(\dot{S} _{15}))=v_{10}(\dot{S} _{15}-1), f_1(v_3(\dot{S} _{15}))=v_{3}(\dot{G}^1_{15}) \\
   & f_1(v_4(\dot{S} _{15}))=v_{7}(\dot{G}^1_{15}), f_1(v_5(\dot{S} _{15}))=v_{1}(\dot{G}^1_{15}), f_1(v_6(\dot{S} _{15}))=v_{9}(\dot{G}^1_{15}) \\
   & f_1(v_7(\dot{S} _{15}))=v_{14}(\dot{G}^1_{15}), f_1(v_8(\dot{S} _{15}))=v_{2}(\dot{G}^1_{15}), f_1(v_9(\dot{S} _{15}))=v_{12}(\dot{G}^1_{15}) \\
   &f_1(v_{10}(\dot{S} _{15}))=v_{8}(\dot{G}^1_{15}), f_1(v_{11}(\dot{S} _{15}))=v_{5}(\dot{G}^1_{15}), f_1(v_{12}(\dot{S} _{15}))=v_{6}(\dot{G}^1_{15}),  \\
   & f_1(v_{13}(\dot{S} _{15}))=v_{13}(\dot{G}^1_{15}), f_1(v_{14}(\dot{S} _{15}))=v_{4}(\dot{G}^1_{15}), f_1(v_{15}(\dot{S} _{15}))=v_{1}(\dot{G}^1_{15}).
\end{align*}

\begin{figure}
  \centering
  \subfigure[$GQ(2,2)$ ]{
  \includegraphics[width=6cm]{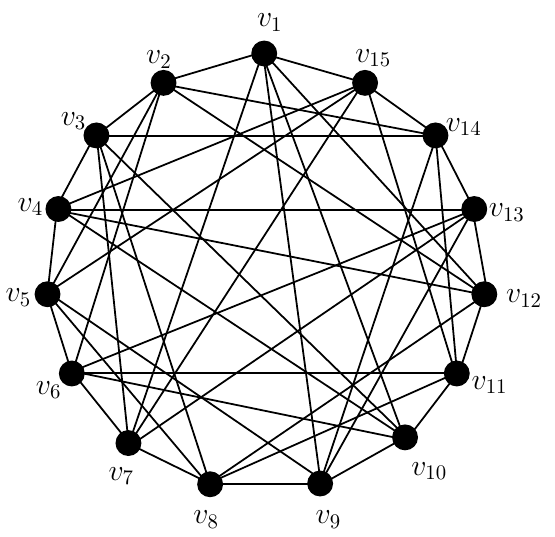}}
    \subfigure[$\dot{S}_{15}$]{
  \includegraphics[width=6cm]{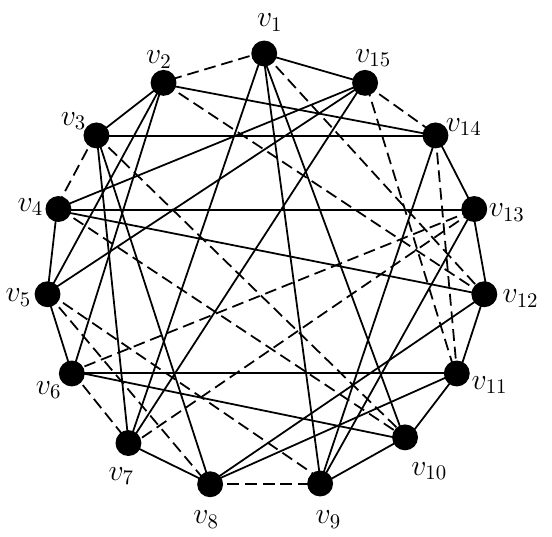}}
  \caption{$GQ(2,2)$ and $\dot{S}_{15}$}\label{(15,6,1,1,-1)}
\end{figure}

\begin{figure}
  \centering
  \subfigure[$\dot{G}^1_{15}$ ]{
  \includegraphics[width=4.5cm]{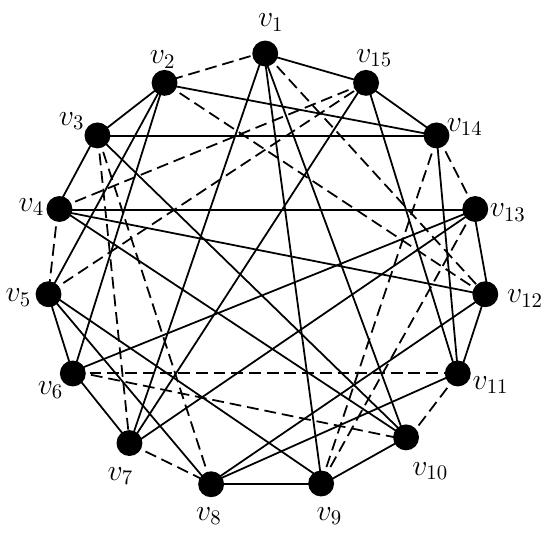}}
    \subfigure[$\dot{G}^2_{15}$]{
  \includegraphics[width=4.5cm]{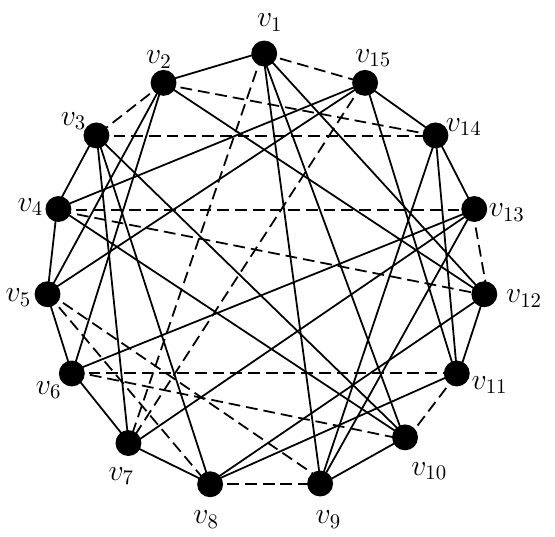}}
    \subfigure[$\dot{G}^3_{15}$ ]{
  \includegraphics[width=4.5cm]{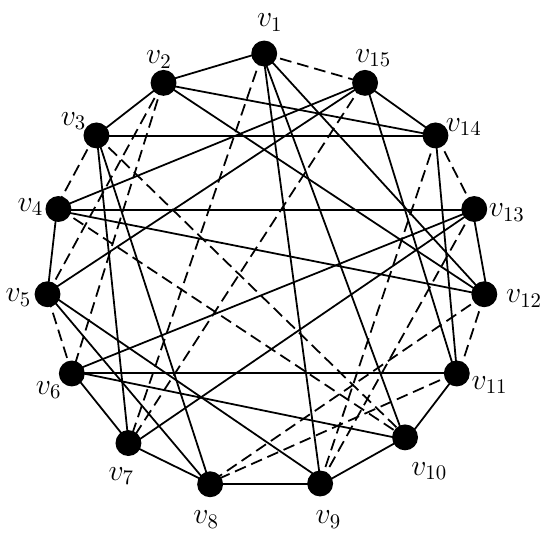}}
    \subfigure[$\dot{G}^4_{15}$]{
  \includegraphics[width=4.5cm]{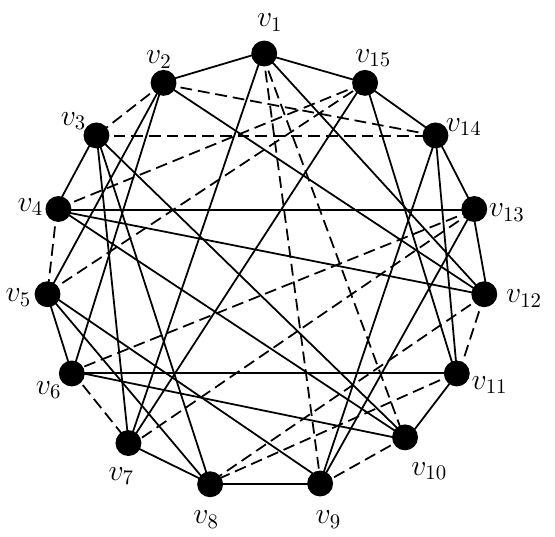}}
    \subfigure[$\dot{G}^5_{15}$ ]{
  \includegraphics[width=4.5cm]{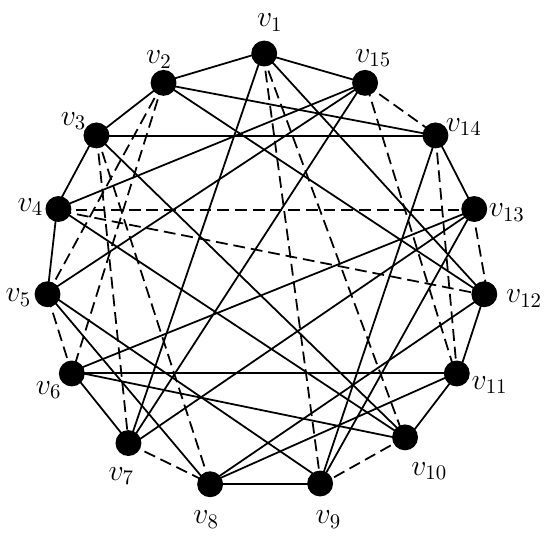}}
  \caption{The other signed graphs which satisfy parameters $(15,6,1,1,-1)$}\label{G15-12345}
\end{figure}

If we choose $v_1v_7v_{15}$ or $v_1v_9v_{10}$ to be the $-K_3$, we also get two signed graphs $\dot{G}^2_{15}$, $\dot{G}^3_{15}$ or $\dot{G} ^4_{15}$, $\dot{G}^5_{15}$, respectively. We show them in Figure \ref{G15-12345}. We also have $\dot{G} ^2_{15}$  is isomorphic to  $\dot{G}^3_{15}$ and $\dot{G} ^4_{15}$  is isomorphic to $\dot{G}^5_{15}$. The corresponding isomorphic mapping are $f_2:V(\dot{G}^2_{15})\rightarrow V(\dot{G}^3_{15})$ and $f_3:V(\dot{G}^4_{15})\rightarrow V(\dot{G}^5_{15})$, respectively. They satisfy that
\begin{equation*}
\begin{aligned}
 &f_2(v_{1}(\dot{G}^2_{15}))=v_{14}(\dot{G}^3_{15}), f_2(v_{2}(\dot{G}^2_{15}))=v_{2}(\dot{G}^3_{15}), f_2(v_{3}(\dot{G}^2_{15}))=v_{5}(\dot{G}^3_{15}),\\ &f_2(v_{4}(\dot{G}^2_{15}))=v_{4}(\dot{G}^3_{15}), f_2(v_{5}(\dot{G}^2_{15}))=v_{12}(\dot{G}^3_{15}), f_2(v_{6}(\dot{G}^2_{15}))=v_{1}(\dot{G}^3_{15}), \\ &f_2(v_{7}(\dot{G}^2_{15}))=v_{9}(\dot{G}^3_{15}), f_2(v_{8}(\dot{G}^2_{15}))=v_{8}(\dot{G}^3_{15}), f_2(v_{9}(\dot{G}^2_{15}))=v_{11}(\dot{G}^3_{15}), \\ &f_2(v_{10}(\dot{G}^2_{15}))=v_{15}(\dot{G}^3_{15}), f_2(v_{11}(\dot{G}^2_{15}))=v_{7}(\dot{G}^3_{15}), f_2(v_{12}(\dot{G}^2_{15}))=v_{3}(\dot{G}^3_{15}), \\ &f_2(v_{13}(\dot{G}^2_{15}))=v_{10}(\dot{G}^3_{15}), f_2(v_{14}(\dot{G}^2_{15}))=v_{6}(\dot{G}^3_{15}), f_2(v_{15}(\dot{G}^2_{15}))=v_{13}(\dot{G}^3_{15}).
 \end{aligned}
\end{equation*}
and
\begin{equation*}
\begin{aligned}
&f_3(v_{1}(\dot{G}^4_{15}))=v_{12}(\dot{G}^5_{15}), f_3(v_{2}(\dot{G}^4_{15}))=v_{2}(\dot{G}^5_{15}), f_3(v_{3}(\dot{G}^4_{15}))=v_{5}(\dot{G}^5_{15}), \\ &f_3(v_{4}(\dot{G}^4_{15}))=v_{15}(\dot{G}^5_{15}), f_3(v_{5}(\dot{G}^4_{15}))=v_{14}(\dot{G}^5_{15}), f_3(v_{6}(\dot{G}^4_{15}))=v_{3}(\dot{G}^5_{15}), \\ &f_3(v_{7}(\dot{G}^4_{15}))=v_{8}(\dot{G}^5_{15}), f_3(v_{8}(\dot{G}^4_{15}))=v_{9}(\dot{G}^5_{15}), f_3(v_{9}(\dot{G}^4_{15}))=v_{13}(\dot{G}^5_{15}), \\ &f_3(v_{10}(\dot{G}^4_{15}))=v_{4}(\dot{G}^5_{15}), f_3(v_{11}(\dot{G}^4_{15}))=v_{10}(\dot{G}^5_{15}), f_3(v_{12}(\dot{G}^4_{15}))=v_{1}(\dot{G}^5_{15}), \\ &f_3(v_{13}(\dot{G}^4_{15}))=v_{7}(\dot{G}^5_{15}), f_3(v_{14}(\dot{G}^4_{15}))=v_{6}(\dot{G}^5_{15}), f_3(v_{15}(\dot{G}^4_{15}))=v_{11}(\dot{G}^5_{15}).
 \end{aligned}
\end{equation*}
Moreover, we have $\dot{S} _{15}$ is isomorphic to  $\dot{G} ^2_{15}$ and $\dot{G} ^4_{15}$, respectively. The corresponding isomorphic mapping are $f_4:V(\dot{S} _{15})\rightarrow V(\dot{G}^2_{15})$ and $f_5:V(\dot{S} _{15})\rightarrow V(\dot{G}^4_{15})$, respectively. They satisfy that
\begin{equation*}
\begin{aligned}
 &f_4(v_{1}(\dot{S} _{15}))=v_{1}(\dot{G}^2_{15}), f_4(v_{2}(\dot{S} _{15}))=v_{7}(\dot{G}^2_{15}), f_4(v_{3}(\dot{S} _{15}))=v_{8}(\dot{G}^2_{15}),\\ &f_4(v_{4}(\dot{S} _{15}))=v_{5}(\dot{G}^2_{15}), f_4(v_{5}(\dot{S} _{15}))=v_{6}(\dot{G}^2_{15}), f_4(v_{6}(\dot{S} _{15}))=v_{13}(\dot{G}^2_{15}), \\ &f_4(v_{7}(\dot{S} _{15}))=v_{12}(\dot{G}^2_{15}), f_4(v_{8}(\dot{S} _{15}))=v_{11}(\dot{G}^2_{15}), f_4(v_{9}(\dot{S} _{15}))=v_{10}(\dot{G}^2_{15}), \\ &f_4(v_{10}(\dot{S} _{15}))=v_{9}(\dot{G}^2_{15}), f_4(v_{11}(\dot{S} _{15}))=v_{14}(\dot{G}^2_{15}), f_4(v_{12}(\dot{S} _{15}))=v_{15}(\dot{G}^2_{15}), \\ &f_2(v_{13}(\dot{S} _{15}))=v_{4}(\dot{G}^2_{15}), f_4(v_{14}(\dot{S} _{15}))=v_{3}(\dot{G}^2_{15}), f_4(v_{15}(\dot{S} _{15}))=v_{2}(\dot{G}^2_{15}).
 \end{aligned}
\end{equation*}
and
\begin{align*}
&f_5(v_{1}(\dot{S} _{15}))=v_{1}(\dot{G}^4_{15}), f_5(v_{2}(\dot{S} _{15}))=v_{9}(\dot{G}^4_{15}), f_5(v_{3}(\dot{S} _{15}))=v_{14}(\dot{G}^4_{15}), \\ &f_5(v_{4}(\dot{S} _{15}))=v_{3}(\dot{G}^4_{15}), f_5(v_{5}(\dot{S} _{15}))=v_{8}(\dot{G}^4_{15}), f_5(v_{6}(\dot{S} _{15}))=v_{5}(\dot{G}^4_{15}), \\ &f_5(v_{7}(\dot{S} _{15}))=v_{15}(\dot{G}^4_{15}), f_5(v_{8}(\dot{S} _{15}))=v_{11}(\dot{G}^4_{15}), f_5(v_{9}(\dot{S} _{15}))=v_{12}(\dot{G}^4_{15}), \\ &f_5(v_{10}(\dot{S} _{15}))=v_{2}(\dot{G}^4_{15}), f_5(v_{11}(\dot{S} _{15}))=v_{6}(\dot{G}^4_{15}), f_5(v_{12}(\dot{S} _{15}))=v_{10}(\dot{G}^4_{15}), \\ &f_5(v_{13}(\dot{S} _{15}))=v_{4}(\dot{G}^4_{15}), f_5(v_{14}(\dot{S} _{15}))=v_{13}(\dot{G}^4_{15}), f_5(v_{15}(\dot{S} _{15}))=v_{7}(\dot{G}^4_{15}).
 \end{align*}
Therefore, these six signed graphs are isomorphic to each other and we choose  $\dot{G}=\dot{S} _{15}$.
\end{proof}


\begin{lemma}
  Let $\dot{G}\in \mathcal{C}_1\bigcup\mathcal{C}_4\bigcup\mathcal{C}_5$ be a connected non-complete 6-regular and 2 net-regular SRSG, then $(a,b)\neq(2,1)$. And if $(a,b)=(3,1)$, then $\dot{G}=\dot{S}^1_{15}$ (Figure \ref{S15}) has parameters $(15,6,3,1,-2)$.
\end{lemma}

\begin{proof}
  If $(a,b)=(2,1)$, there are only two positive walks of length 2 between any two positively adjacent vertices in $\dot{G}$ since $r=6$. Suppose $v_iv_j$ is a negative edge. If there are three positive and two negative walks of length 2 between $v_i$ and $v_j$, it must exist $v_k\in N(v_i)\bigcap N(v_j)$  such that $v_k\overset{+} {\sim}v_i,v_j$. Then there will exist two negative walks length 2 between $v_i$ and $v_k$. This contradicts $a=2$. Therefore, there is only one positive walk of length 2 between any two negatively adjacent vertices. And the corresponding triangle is $-K_3$. That is to say that $G^-$ is the union of $K_3$. Then $3|n$.  we get the possible parameters sets are $(9,6,2,1,-6)$, $(10,6,2,1,-4)$, $(11,6,2,1,-3)$, $(13,6,2,1,-2)$ and $(19,6,2,1,-1)$ by equation \eqref{6-2-abcn}. Since $3|n$ and  $n\geq 15$ by $b=1$, there is no SRSG with $(a,b)=(2,1)$.

  If $(a,b)=(3,1)$, then there are only three positive walks of length 2 between any two positively adjacent vertices in $\dot{G}$. Then  $G^+$ is the union of $K_5$ by $a=3$ and
   we also have $G^-$ is the union of $K_3$ by the same reason as $(a,b)=(2,1)$. Therefore, we have $15|n$.  And by equation \eqref{6-2-abcn}, we get the possible parameters sets are $(23,6,3,1,-1)$, $(15,6,3,1,-2)$ and $(11,6,3,1,-4)$. So we only need to consider $(15,6,3,1,-2)$. It is easy to see that the corresponding SRSG is the signed graph $\dot{S}^1_{15}$ shown in Figure \ref{S15}.
\end{proof}

\begin{figure}
  \centering
  \includegraphics[width=9cm]{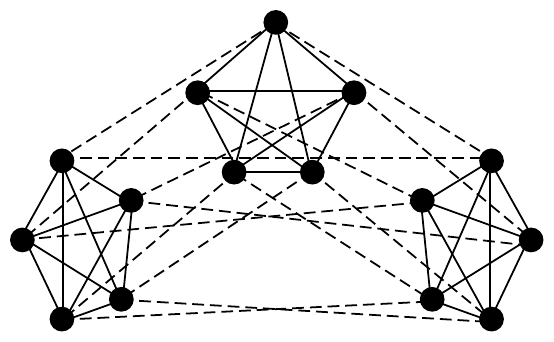}
  \caption{$\dot{S}^1_{15}$}\label{S15}
\end{figure}

\textbf{Secondly, we consider that $\dot{G}$ only has balanced triangles. }

\begin{lemma}
  Let $\dot{G}\in \mathcal{C}_1\bigcup\mathcal{C}_4\bigcup\mathcal{C}_5$ be a connected non-complete 6-regular and 2 net-regular SRSG. If $\dot{G}$ only has balanced triangles, then

  (1) $a=0,1,2,3$ and $b=-2,0$;

  (2)  $a=0$ ($b=0$, respectively) implies that two positively (negatively, respectively) adjacent vertices don't have  common neighbours.
\end{lemma}

\begin{proof}
Since $\dot{G}$ only has balanced triangles, there is no negative walk of length 2 between any two positively adjacent vertices and no positive walk of length 2 between any two negatively adjacent vertices. So we have $a\geq 0$ and $b\leq 0$ . Suppose $v_iv_j$ is a positive edge of $\dot{G}$.

If $a=5$, then there are five positive walks of length 2 between any two  positively adjacent vertices. Suppose that $N(v_i)\bigcap N(v_j)=\{v_k,v_l,v_m,v_s,v_t\}$ and $v_k,v_l,v_m\overset{+} {\sim}v_i,v_j$ and $v_s,v_t\overset{-} {\sim}v_i,v_j$. For the  positive edge $v_iv_k$, it must have  $v_k\overset{-} {\sim}v_s,v_t$. But this leads to $d^-(v_s)=d^-(v_t)=3$. A contradiction.

If $a=4$, then there are four positive walks of length 2 between any two  positively adjacent vertices.  Suppose that $N(v_i)\bigcap N(v_j)=\{v_k,v_l,v_m,v_s\}$ and $v_t$ is the remaining neighbour of $v_i$. It must have one or two vertices of $\{v_k,v_l,v_m,v_s\}$ are negatively adjacent to $v_i,v_j$. If  two vertices are negatively adjacent to both $v_i$ and $v_j$, we suppose  $v_m,v_s\overset{-} {\sim}v_i,v_j$ without loss of generality. So we have $v_k, v_l$ are negatively adjacent to one of $v_m,v_s$ by $a=4$, then it is contradictory to the vertex net-degree. If  only one vertex is negatively adjacent to both $v_i$ and $v_j$, we suppose $v_s\overset{-} {\sim}v_i,v_j$. Then we have $v_t\overset{-} {\sim}v_i$. Since  $d^-(v_s)=2$, we have $v_k, v_l\overset{-} {\nsim} v_s$. So it must have $v_k, v_l\overset{-}{\sim} v_t$ to satisfy $a=4$ for positive edges  $v_iv_k$ and $v_iv_l$. But it leads to
$d^-(v_t)>2$. A contradiction.

By Lemma \ref{le2} and vertex negative degree,  there are at most two negative walks of length 2 between two negatively adjacent vertices. Then we have $b\in \{0,-1,-2\}$. But there will be one unbalanced triangle if  $b=-1$. Therefore, $b\in \{-2,0\}$.

Since two positively (negatively, respectively) adjacent vertices don't have negative (positive, respectively) walks of length 2, they don't have common neighbours if $a=0$ ($b=0$, respectively).

\end{proof}

If $\dot{G}$ has only balanced triangles, we have $a\in \{0,1,2,3\}$ and $b\in \{0,-2\}$ by the above Lemma. And $(a,b)\neq(0,0)$ if $\dot{G}$ has triangles. And it is obvious that $(a,b)\neq(0,-2)$ since there will be one positive walk of length 2 between two positively adjacent vertices.

If $(a,b)=(1,0)$, the possible parameters sets are $(13,6,1,0,-1)$, $(10,6,1,0,-2)$, $(9,6,1,0,-3)$ and $(8,6,1,0,-6)$ by equation \eqref{6-2-abcn}. Since $a=1$ and $\dot{G}$ contains only balanced triangles,  there is only one positive walk of length 2 between any two positively adjacent vertices and any two negatively adjacent vertices have no common neighbours. Let  $v_1v_2$ be a negative edge of $\dot{G}$, $N(v_1)=\{v_2, v_3, v_4, v_5, v_6, v_7\}$ and $N(v_2)=\{v_1,v_8, v_9, v_{10}, v_{11}, v_{12}\}$. Suppose $v_3\overset{-} {\sim}v_1$,\; $v_4, v_5, v_6, v_7\overset{+} {\sim}v_1$,\;  $v_8\overset{-} {\sim}v_2$ and $ v_9, v_{10}, v_{11}, v_{12}\overset{-} {\sim}v_2$. Obviously, we have $n\geq 12$. So we only need to consider parameters $(13,6,1,0,-1)$. Since vertices $v_1$ and $v_8$ are not adjacent,  $v_8$ has only three common neighbours with $v_1$ by $c=-1$. And  $v_8$ is not adjacent to $v_9, v_{10}, v_{11}, v_{12}$, then we have $n> 13$. A contradiction.
%

If $(a,b)=(2,0)$, then two positive walks of length 2 between any two positively adjacent vertices are $+P_3$ and any two negatively adjacent vertices have no common neighbours. Let $v_iv_j$ be a positive edge of $\dot{G}$, $N^+(v_i)=\{v_j, v_k, v_l, v_m\}$ and $N^+(v_j)=\{v_i, v_l, v_m, v_s\}$. Then we have $ v_k\overset{+} {\sim} v_l, v_m$ and  $ v_s\overset{+} {\sim} v_l, v_m$ by the positive edges $v_iv_k$ and $v_jv_s$. And by the  positive edge $v_lv_k$, we have $v_k\overset{+} {\sim}v_s$. Now every vertex has positive degree 4. So $G^+$ is the union of a 4-regular graph with order 6 as shown in Figure \ref{4re-order6}. Then $6|n$. We get that the possible parameters are $(17,6,2,0,-1)$, $(12,6,2,0,-2)$ and $(9,6,2,0,-5)$ by equation \eqref{6-2-abcn}. It is obvious that we only need to consider $(12,6,2,0,-2)$. Since there are four positive walks of length 2 between $v_j$ and $v_k$, they need six negative walks of length 2 by $c=-2$. But this is contradictory to $r=6$.

\begin{figure}
  \centering
  \includegraphics[width=4cm]{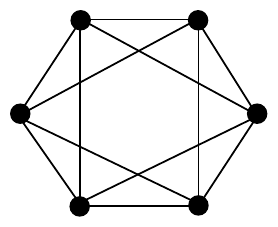}
  \caption{ The 4-regular graph with order 6}\label{4re-order6}
\end{figure}

If $(a,b)=(3,0)$, then there are three positive walks of length 2 between any two positively adjacent vertices and they are $+P_3$. And any two negatively adjacent vertices have no common neighbours. Let $v_iv_j$ be a positive edge of $\dot{G}$ and $N(v_i)\bigcap N(v_j)=\{v_k, v_l, v_m\}$.  Then we have $ v_k\overset{+} {\sim} v_l, v_m$ and  $ v_l\overset{+} {\sim} v_m$ by the positive edges $v_iv_k$ and $v_iv_l$. So $G^+$ is the union of $K_5$. Therefore, we have $5|n$. By equation \eqref{6-2-abcn}, the possible parameters sets are: $(21,6,3,0,-1)$, $(14,6,3,0,-2)$, $(9,6,3,0,-7)$. It is obvious that they are not true.

If $(a,b)=(1,-2)$, the possible parameters sets are  $(9,6,1,-2,-1)$ and $(8,6,1,-2,-2)$ by equation \eqref{6-2-abcn}. Since $(a,b)=(1,-2)$, two adjacent vertices in underlying graph $G$ have one or two  common neighbours. There are four 6-regular graphs with order 9 (Figure \ref{9-6}) and one  6-regular graph with order 8 ($G_8$ of Figure \ref{8-6.864-4-6}). We can get that none of them satisfies this condition by some computations.

If $(a,b)=(2,-2)$, the possible parameters sets are $(13,6,2,-2,-1)$, $(10,6,2,-2,-2)$, $(9,6,2,-2,-3)$ and $(8,6,2,-2,-6)$ by equation \eqref{6-2-abcn}. For the parameters  $(10,6,2,-2,-2)$, any two adjacent vertices must have two common neighbours in the underlying graph. There are 21 6-regular graphs with order 10 (Figure \ref{10-6}). We can get that none of them satisfies this condition by some simple computations.
 The corresponding underlying graphs of the first and last two sets  are strongly regular graphs with $(13,6,2,3)$, $(9,6,2,5)$ and $(8,6,2,6)$. There is no  strongly regular graph with parameters $(9,6,2,5)$ or $(8,6,2,6)$ by equation \eqref{srg}. The first strongly regular graph is the Paley Graph $P_{13}$ with order 13 which shown  in Figure \ref{SP13}. Let $\dot{G}$ be a 6-regular and 2 net-regular SRSG with sign function $\sigma$ and parameters $(13,6,2,-2,-1)$, then its underlying graph is $P_{13}$. Since $\dot{G}$ is inhomogeneous, we suppose $\sigma(v_7v_8)=-1$ without loss of generality. If $\sigma(v_8v_{11})=\sigma(v_7v_4)=-1$ and $\sigma(v_8v_4)=\sigma(v_7v_{11})=1$, we can get the subgraph of  $\dot{G}$ as show in Figure \ref{SP13} by considering the  negative edge $v_8v_{11}$ and every emerging negative edge in turn. Now, we have $d^-(v_{7})=3$ which is contradictory to the vertex negative degree. If $\sigma(v_8v_{11})=\sigma(v_7v_4)=1$ and $\sigma(v_8v_4)=\sigma(v_7v_{11})=-1$, we have $d^-(v_{4})=3$ by the same way. Also a contradiction.

\begin{figure}
  \centering
  \subfigure [$P_{13}$]{
    \includegraphics[width=7cm]{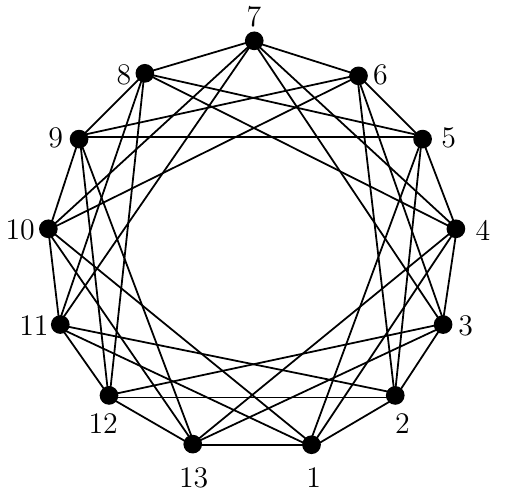}}
      \subfigure[$SP_{13}$]{
    \includegraphics[width=7cm]{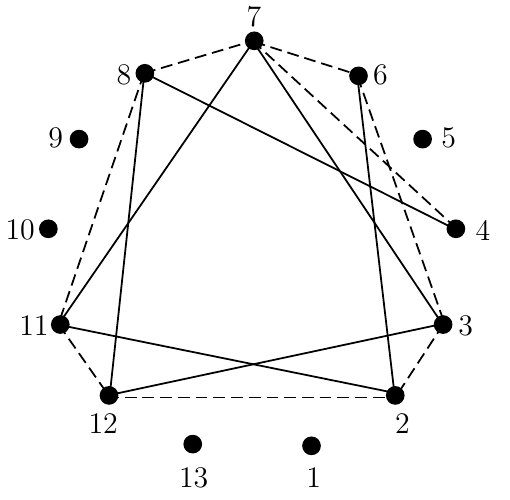}}
    \caption{Paley Graph with order 13 and the part of signed paley Graph with order 13 }\label{SP13}
\end{figure}

%

\begin{lemma}
  Let $\dot{G}\in \mathcal{C}_1\bigcup\mathcal{C}_4\bigcup\mathcal{C}_5$ be a connected non-complete 6-regular and 2 net-regular SRSG. If it contains only balanced triangles, then  $(a,b)\neq(3,-2)$.
\end{lemma}

\begin{proof}
   If $(a,b)=(3,-2)$, then there are three positive walks of length 2 between any two positively adjacent vertices and two negative walks of length 2 between any two negatively adjacent vertices. Let $v_iv_j$ be a negative edge of $\dot{G}$. Suppose that $N(v_i)=\{v_j,v_k,v_l,v_{m_1},v_{s_1},v_{t_1}\}$, $N(v_j)=\{v_i, v_k,v_l,v_{m_2},v_{s_2},v_{t_2}\}$ and  $v_k\overset{+} {\sim}v_i$,  $v_k\overset{-} {\sim}v_j$, $v_l\overset{-} {\sim}v_i$,  $v_l\overset{+} {\sim}v_j$. We consider the negative edge $v_iv_l$, then we have $v_k\overset{-} {\sim}v_l$ or $v_l$ is negatively adjacent to one of $v_{m_1},v_{s_1},v_{t_1}$.

   If $v_k\overset{-} {\sim}v_l$, we consider the positive edges $v_iv_{t_1}$ and $v_iv_{s_1}$. Then it must have $v_{t_1}\overset{+} {\sim}v_k, v_{m_1},v_{s_1}$ for $v_iv_{t_1}$. Now there are three positive walks of length 2 between $v_i$ and $v_k$, then $v_{s_1}\nsim v_k$. So there are at most two positive walks of length 2 between $v_i$ and $v_{s_1}$. A contradiction.

   For the second case, we suppose $v_l\overset{-} {\sim}v_{m_1}$ without loss of generality. It also must have $v_{t_1}\overset{+} {\sim}v_k, v_{m_1},v_{s_1}$ for the positive edge $v_iv_{t_1}$. And then we have $v_{s_1}\overset{+} {\sim}v_k, v_{m_1}$ for the positive edge $v_iv_{s_1}$. For the negative edge  $v_jv_k$, we can suppose $v_k\overset{-} {\sim}v_{m_2}$ without loss of generality. In a similar way, we have $v_{t_2}\overset{+} {\sim}v_l, v_{m_2},v_{s_2}$ and $v_{s_2}\overset{+} {\sim}v_l, v_{m_2}$  for the positive edges $v_jv_{t_2}$, $v_jv_{s_2}$.

   Now we consider the vertex $v_k$, it remains a positive neighbour, denoted by $v_h$. Then there must be three positive walks of length 2 between $v_k$ and $v_h$. Since $v_k\nsim v_{m_1}$ and $d^+(v_i)=d^+(v_{s_1})=d^+(v_{t_1})=4$ and $d(v_j)=6$, we have $v_h\nsim v_i,v_j, v_{s_1},v_{t_1}$. So there can't be three positive walks of length 2 between $v_k$ and $v_h$. A contradiction.
\end{proof}

\textbf{Finally, we consider that $\dot{G}$ does not contains triangles.}

If there is no triangle in $\dot{G}$, then $(a,b)=(0,0)$. So the possible parameters sets are $(9,6,0,0,-1)$ and $(8,6,0,0,-2)$. In this case, any two adjacent vertices in $G$  have no common neighbours. None 6-regular graphs with order 8 ($G_8$ of Figure \ref{8-6.864-4-6}) and 9 (Figure \ref{9-6}) satisfy this condition.

To sum up the above contents, we have the following Theorem.

\begin{theorem}
  There are six connected 6-regular and 2 net-regular SRSGs:  $\dot{S}^2_8$, $\dot{S}^3_8$, $\dot{S}_9$, $\dot{S}^1_9$, $\dot{S}_{15}$ and $\dot{S}^1_{15}$.
\end{theorem}

\subsection{The 6-regular SRSGs with net-degree 0}

 Let $\dot{G}$ be a connected 6-regular and 0 net-regular SRSG, then both $G^+$ and $G^-$ are 3-regular. And we have

\begin{equation}\label{6-0-abcn}
3(a+b)+c(n-7)+6=0
\end{equation}
 by equation \eqref{def}. Moreover, we only need to consider  $a\geq b$ by Lemma \ref{le1} since both $\dot{G}$ and $-\dot{G}$ are 0 net-regular.

 If $\dot{G}\in \mathcal{C}_2$, we have $A^2_{\dot{G}}+bA_{\dot{G}}-6I=0$ by equation \eqref{def}. But this is impossible in this case since $\dot{G}$ has the net-degree 0 as an eigenvalue.

If $\dot{G}\in \mathcal{C}_3$, then $\dot{G}$ is complete \cite{KS}. But there is no 3-regular graph with order 7. Then $G^+$ and $G^-$ don't exist. By the same way, there is no complete SRSG in $\mathcal{C}_1$.

Now we consider the connected non-complete SRSG $\dot{G}\in \mathcal{C}_1\bigcup\mathcal{C}_4\bigcup\mathcal{C}_5$.

\begin{lemma}\label{6-0}
  If $\dot{G}\in \mathcal{C}_1\bigcup\mathcal{C}_4\bigcup\mathcal{C}_5$ is a connected non-complete 6-regular and 0 net-regular SRSG, then

  (1) $a,b<5$;

  (2) If $a=4$ ($b=4$, respectively), then $b\leq-3$ ($a\leq-3$, respectively);

  (3) $a,b\neq-3$;

  (4) $a,b\geq-4$. And if $b=-4$ ($a=-4$, respectively), then $a\geq0$ ($b\geq0$, respectively);

  (5) If $b=3$ ($a=3$, respectively), then $a\leq1$ ($b\leq1$, respectively);

  (6) $a,b\neq-2$.

  (7) If $b=-1$ ($a=-1$, respectively), then $a\leq1$ ($b\leq1$, respectively).

\end{lemma}

\begin{proof}
  (1) Let $v_iv_j$ be a positive edge of $\dot{G}$. If $a=5$, then there are five positive walks of length 2 between $v_i$ and $v_j$.  Suppose $N(v_i)\bigcap N(v_j)=\{v_k,v_l,v_m,v_s,v_t\}$ and $v_k,v_l\overset{+} {\sim}v_i,v_j$ and $v_m,v_s,v_t\overset{-} {\sim}v_i,v_j$. By the positive edges $v_iv_k$ and $v_iv_l$, we have $v_k\overset{+} {\sim}v_l$, $v_k\overset{-} {\sim}v_m,v_s,v_t$ and $v_l\overset{-} {\sim}v_m,v_s,v_t$. Then $d^-(v_m),d^-(v_s),d^-(v_t)\geq4$. A contradiction. In a similar way, we have $b\neq5$.

  (2) Let $v_iv_j$ be a positive edge of $\dot{G}$ and $a=4$. Then there are four positive walks of length 2 between $v_i$ and $v_j$.  Suppose $N(v_i)\bigcap N(v_j)=\{v_k,v_l,v_m,v_s\}$. There are two cases: Two vertices of $\{v_k,v_l,v_m,v_s\}$ are positively adjacent to $v_i,v_j$; Only one vertex of $\{v_k,v_l,v_m,v_s\}$ is positively adjacent to $v_i,v_j$. For the second case, we suppose $v_k\overset{+} {\sim}v_i,v_j$  and  $v_l, v_m,v_s\overset{-} {\sim}v_i,v_j$. Then we have $v_{t_1}\overset{+} {\sim}v_i$ and $v_{t_2}\overset{+} {\sim}v_j$, where $v_{t_1}$ and $v_{t_2}$ are the remaining neighbours of $v_i$ and $v_j$, respectively. Then it must have $v_{t_1}\overset{-} {\sim}v_l,v_m,v_s$ and $v_{t_2}\overset{-} {\sim}v_l,v_m,v_s$ by positive edges $v_i v_{t_1}$ and $v_j v_{t_2}$. But $d^-(v_l)=d^-(v_m)=d^-(v_s)=4$ now. A contradiction. Then we only consider the first case.

   We suppose that $v_k,v_l\overset{+} {\sim}v_i,v_j$  and  $v_m,v_s\overset{-} {\sim}v_i,v_j$ and consider the positive edges $v_i v_k$ and $v_i v_l$. For the positive edge $v_i v_k$, $v_k$ cannot be adjacent to $v_m$ and $v_s$ at the same time. Otherwise, there will be one vertex which is incident with four negative edges since $v_l$ is adjacent to at least one of $v_m,v_s$. A contradiction. If $v_k\overset{-} {\sim}v_m$, there are two negative walks of length 2 between $v_i$ and $v_m$.  There must be the other two  negative walks of length 2 between $v_i$ and $v_m$ by Lemma \ref{le2}, then we have $b\leq-3$. For $v_k\overset{-} {\sim}v_s$, we have the same result.

   By Lemma \ref{le1}, we have $a\leq-3$ if $b=4$.

   (3) Let $v_iv_j$ be a positive edge of $\dot{G}$. If $a=-3$, then there are four negative walks and one positive walk of length 2 between $v_i$ and $v_j$.  Let $v_k,v_l,v_m,v_s, v_t$ be the common neighbours of $v_i,v_j$ and suppose that $v_k, v_l\overset{+} {\sim}v_i$, $v_k, v_l\overset{-} {\sim}v_j$, $v_m\overset{-} {\sim}v_i, v_j$, $v_s, v_t\overset{-} {\sim}v_i$, $v_s, v_t\overset{+} {\sim}v_j$. By the positive edges $v_i v_k$, $v_iv_l$, we have $v_k\sim v_l,v_m,v_s, v_t$ and $v_l\sim v_m,v_s, v_t$. And by the positive edges $v_j v_s$, $v_jv_t$, we have $v_m\sim v_s, v_t$ and $v_s\sim v_t$. So $\dot{G}$ is complete.

    By Lemma \ref{le1}, we have  $b\neq-3$.

   (4) There are at most four negative walks of length 2 between any two adjacent vertices in $\dot{G}$ by Lemma \ref{le2} and $\rho=0$. Therefore, $a,b\geq-4$.

    Suppose that $v_iv_j$ is a negative edge of $\dot{G}$ and $b=-4$. Then there are four negative walks of length 2 between $v_i$ and $v_j$. Let $v_k,v_l,v_m,v_s$ be the common neighbours of $v_i,v_j$. And we suppose that $v_k, v_l\overset{-} {\sim}v_i$, $v_k, v_l\overset{+} {\sim}v_j$ and $v_m, v_s\overset{+} {\sim}v_i$,  $v_m, v_s\overset{-} {\sim}v_j$. We consider the negative edge $v_jv_s$, then $v_s$ must be negatively adjacent to at least one of $v_k,v_l$.  So there will be two positive walks of length 2 between $v_j$ and $v_l$ or  $v_j$ and $v_k$. Therefore, $a\geq0$.

    By Lemma \ref{le1}, we have $b\geq 0$ if $a=-4$.

   (5) Suppose that $v_iv_j$ is a negative edge of $\dot{G}$ and $b=3$. Then there are three positive walks of length 2 between $v_i$ and $v_j$ and at least one walk consists of two positive edges. That is to say there exists  at least one unbalanced triangle. Therefore, there will be two positively adjacent vertices with two negative walks of length 2 between them. Then  $a\leq1$.

     By Lemma \ref{le1}, we have $b\leq1$ if $a=3$.

   (6) Let $v_iv_j$ be a negative edge of  $\dot{G}$. Then there are two negative walks of length 2 between $v_i$ and $v_j$ if $b=-2$. Suppose that $N(v_i)\bigcap N(v_j)=\{v_k,v_l\}$ such that  $v_k\overset{+} {\sim}v_i$, $v_k\overset{-} {\sim}v_j$, $v_l\overset{-} {\sim}v_i$, $v_l\overset{+} {\sim}v_j$ and $v_m$ is the remaining negative neighbour of  $v_i$. Then it must have  $v_m\overset{+} {\sim}v_l$ by $b=-2$ and $v_m\nsim v_j$. But the situation of these two negative walks of length 2 between $v_i$ and $v_l$ is contradict to Lemma \ref{le2}.

   We also have $a\neq-2$ by Lemma \ref{le1}.

   (7) Let $v_iv_j$ be a negative edge of $\dot{G}$ and $b=-1$. Then there are two negative walks and one positive walk of length 2 between $v_i$ and $v_j$. Suppose $N(v_i)\bigcap N(v_j)=\{v_k,v_l,v_m\}$ and $v_k\overset{+} {\sim}v_i$, $v_k\overset{-} {\sim}v_j$, $v_l\overset{-} {\sim}v_i$, $v_l\overset{+} {\sim}v_j$. For vertex $v_m$, we consider negative edge $v_iv_m$ if $v_m\overset{-} {\sim}v_i,v_j$. Since there is one positive walk of length 2  between $v_i$ and $v_m$, they need two negative walks of length 2. Therefore, we have $v_m\overset{+} {\sim}v_l$. But there will be other two negative walks of length 2 between $v_l$ and $v_i$ by the existing  two negative walks of length 2 between them and Lemma \ref{le2}. Then it will lead to $b\leq-3$. A contradiction. So it must be $v_m\overset{+} {\sim}v_i,v_j$. Now we consider positive edge $v_iv_m$. Since there is one negative walk of length 2 between $v_i$ and $v_m$, we have $a\leq1$ by Lemma \ref{le2}.

  By Lemma \ref{le1}, we have $b\leq1$ if $a=-1$.
\end{proof}

%

According to the above Lemma and the condition of $a\geq b$, the values of
$(a,b)$ we need to concern are: $(4,-4)$, $(3,1)$, $(3,0)$,  $(3,-4)$,  $(2,2)$, $(2,1)$, $(2,0)$, $(2,-4)$,  $(1,1)$, $(1,0)$, $(1,-1)$, $(1,-4)$, $(0,0)$, $(0,-1)$, $(0,-4)$, $(-1,-1)$.

\begin{lemma}
Let $\dot{G}\in \mathcal{C}_1\bigcup\mathcal{C}_4\bigcup\mathcal{C}_5$ be  a connected and non-complete 6-regular and 0 net-regular SRSG. If $(a,b)= (4,-4)$, then $\dot{G}$ is isomorphic to $\dot{S}^4_8$ (Figure \ref{864-4-6}) with parameters $(8,6,4,-4,-6)$.
\end{lemma}

\begin{proof}
  We suppose  $v_iv_j$ is a negative edge of  $\dot{G}$. If $(a,b)=(4,-4)$,  then there are four negative walks of length 2 between $v_i$ and $v_j$. Suppose that $N(v_i)=\{v_j, v_k,v_l,v_m,v_s,v_{t_1}\}$,  $N(v_j)=\{v_i, v_k,v_l,v_m,v_s,v_{t_2}\}$ and $v_k, v_l\overset{+} {\sim}v_i$, $v_k, v_l\overset{-} {\sim}v_j$, $v_m, v_s\overset{-} {\sim}v_i$, $v_m, v_s\overset{+} {\sim}v_j$. Then $v_i\overset{+} {\sim}v_{t_1}$ and $v_j\overset{+} {\sim}v_{t_2}$. Since $a=4$, we have $v_{t_1}\overset{+} {\sim}v_k, v_l$, $v_{t_1}\overset{-} {\sim}v_m, v_s$, $v_{t_2}\overset{-} {\sim}v_k, v_l$,  $v_{t_2}\overset{+} {\sim}v_m, v_s$. By considering the edges $v_iv_m$,  $v_iv_s$ and  $v_iv_k$ in turn, we have two cases as shown in Figure \ref{Two cases}. Now, $d(v_i)=d(v_j)=d(v_k)=d(v_l)=d(v_m)=d(v_s)=6$ and we  need one positive walk of length 2 between  $v_{t_1}$ and $v_k$ in these two cases, so it must have $v_{t_1}\overset{-} {\sim}v_{t_2}$. It is easy to see that these two graphs are isomorphic. So $\dot{G}$ is  isomorphic to $\dot{S}^4_8$ (Figure \ref{8-6.864-4-6}) and has parameters $(8,6,4,-4,-6)$.
\end{proof}

\begin{figure}[h]
  \centering
  \includegraphics[width=12cm]{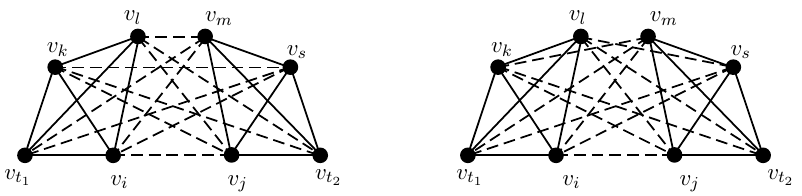}
  \caption{Two cases for  $(a,b)=(4,-4)$}\label{Two cases}
\end{figure}

\begin{figure}[h]
  \centering
  \includegraphics[width=5.5cm]{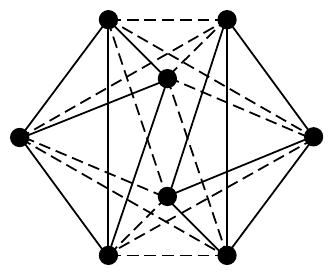}
  \caption{$\dot{S}^4_8$}\label{864-4-6}
\end{figure}

\begin{lemma}\label{6-0-31}
   If $\dot{G}\in \mathcal{C}_1\bigcup\mathcal{C}_4\bigcup\mathcal{C}_5$ is a connected and non-complete 6-regular and 0 net-regular SRSG, then $(a,b)\neq(3,1), (3,0)$.
\end{lemma}

\begin{proof}
If $a=3$, there are only three positive walks of length 2 between any two positively adjacent vertices. And  at least one of these three positive walks is $-P_3$ by the net-degree. So $\dot{G}$ must contain at least one balanced triangle with two negative edges. Denote this triangle by $v_iv_jv_k$ and suppose $v_i\overset{-} {\sim}v_j$, $v_k\overset{-} {\sim}v_j$, $v_i\overset{+} {\sim}v_k$. Since there is one negative walk of length 2 between $v_i$  and $v_j$, there is another common neighbour  of $v_i$ and $v_j$, denoted by $v_l$, such that $v_l\overset{-} {\sim}v_i$ and $v_l\overset{+} {\sim}v_j$ by Lemma \ref{le2}.

If $b=1$, then $v_i$ and $v_j$ need three positive walks of length 2 and only one of these walks is $-P_3$ by vertex net-degree. Suppose that $v_m,v_s\overset{+} {\sim}v_i,v_j$, $v_t\overset{-} {\sim}v_i,v_j$ without loss of generality. But there is one negative walk of length 2 between $v_m$ and $v_i$, which contradicts $a=3$.

If $b=0$, then $v_i$ and $v_j$ need two positive walks of length 2 and at least one of these walks is $+P_3$ by vertex net-degree. It is also contradictory to $a=3$ by the same way as $b=1$.
\end{proof}

 If $\dot{G}\in \mathcal{C}_1\bigcup\mathcal{C}_4\bigcup\mathcal{C}_5$ is a connected and non-complete 6-regular and 0 net-regular SRSG with $(a,b)=(3,-4)$, we have $c(n-7)=-3$ by equation \eqref{6-0-abcn}. So the possible parameters sets are $(10,6,3,-4,-1)$ and $(8,6,3,-4,-3)$. Since $(a,b)=(3,-4)$, two adjacent vertices in $G$ have three or four common neighbours. And $c=-1$ implies that there are three common neighbours between any two non-adjacent vertices in $G$.  We get that all 6-regular graphs with order 8 ( $G_8$ of Figure \ref{8-6.864-4-6}) and order 10 (Figure \ref{10-6}) cannot satisfy this condition by some simple computations.

 \begin{lemma}\label{6-0-22}
  Let $\dot{G}\in \mathcal{C}_1\bigcup\mathcal{C}_4\bigcup\mathcal{C}_5$ be a connected and non-complete 6-regular and 0 net-regular SRSG. If $(a,b)=(2,2)$, then $\dot{G}=\dot{S}_{16}$ (Figure \ref{16,6,2,2,-2}).
\end{lemma}

\begin{proof}
   If $(a,b)=(2,2)$, there are only two positive walks of length 2 between any two adjacent vertices in $\dot{G}$.  Let $v_iv_j$ be a positive edge of $\dot{G}$. We claim that the two positive walks between $v_i$ and $v_j$  must be $+P_3$. Otherwise, suppose $v_m\in N(v_i) \bigcap N(v_j)$ and $v_m\overset{-} {\sim}v_i,v_j$. There is one negative walk of length 2 between  $v_i$ and $v_m$, which is contradictory to $b=2$. Suppose $v_k,v_l\in N(v_i) \bigcap N(v_j)$ and so $v_k,v_l\overset{+} {\sim}v_i,v_j$. We have $v_k\overset{+} {\sim}v_l$ by the positive edge  $v_iv_k$. Therefore, $G^+$ is the union of $K_4$. In a similar way, we get that $G^-$ is the union of $K_4$. So $4|n$.

   Since $(a,b)=(2,2)$, we have $c(n-7)=-18$ by equation \eqref{6-0-abcn}. So the possible parameters sets are: $(25,6,2,2,-1)$,  $(16,6,2,2,-2)$, $(13,6,2,2,-3)$ and $(10,6,2,2,-6)$. Since $4|n$, we only need consider the parameters $(16,6,2,2,-2)$. Suppose vertices $v_1,v_2,v_3,v_4$ form a $-K_4$ in $\dot{G}$ and $v_5,v_6,v_7\in N^+(v_1)$, $v_8,v_9,v_ {10}\in N^+(v_2)$, $v_{11},v_{12},v_{13}\in N^+(v_3)$, $v_{14},v_{15},v_{16}\in N^+(v_4)$. By the above proof, they form four $K_4$, respectively. Since $n=16$ and $G^-$ is the union of $K_4$, $\dot{G}$ has three other $-K_4$ which formed by $v_5,v_6,\dots, v_{15},v_{16}$. Without loss of generality, we suppose $v_5,v_8,v_ {11},v_{14}$ form a $-K_4$, so do  $v_6,v_9,v_{12},v_{15}$ and $v_7,v_{10},v_{13},v_{16}$. Therefore, we get the signed graph $\dot{S}_{16}$ in Figure \ref{16,6,2,2,-2} and it satisfies $c=-2$. Therefore, $\dot{G}$ is isomorphic to $\dot{S}_{16}$.
\end{proof}

\begin{figure}
  \centering
  \includegraphics[width=7cm]{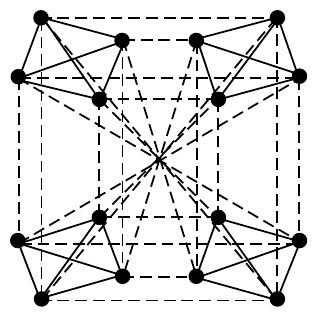}
  \caption{$\dot{S}_{16}$}\label{16,6,2,2,-2}
\end{figure}

\begin{lemma}\label{6-0-21.11}
If  $\dot{G}\in \mathcal{C}_1\bigcup\mathcal{C}_4\bigcup\mathcal{C}_5$ is a connected non-complete 6-regular and 0 net-regular SRSG, then $(a,b)\neq (2,1), (1,1)$.
\end{lemma}

\begin{proof}
   Let $v_iv_j$ be a negative edge of $\dot{G}$. If $b=1$, then there are two cases for $v_i, v_j$: Case 1. There is only one positive  walk of length 2 between  $v_i$ and $v_j$; Case 2. There are three positive  and two negative  walks of length 2 between  $v_i$ and $v_j$.

   We discuss the Case 2 firstly. Suppose $N(v_i)\bigcap N(v_j)=\{v_k, v_l, v_m, v_s, v_t\}$ and $v_k\overset{-} {\sim}v_i, v_j$,\; $v_l\overset{+} {\sim}v_i$, $v_l\overset{-} {\sim}v_j$,\; $v_m\overset{-} {\sim}v_i$, $v_m\overset{+} {\sim}v_j$ and $v_s, v_t\overset{+} {\sim}v_i, v_j$. We have $v_m\overset{-} {\sim}v_k$ for the negative edge $v_iv_m$ since $b=1$. Now there are two positive walks of length 2 between $v_i$ and $v_k$ but no negative walks of length 2. A contradiction.

   Therefore, there is only one positive walk of length 2 between each pair of negatively adjacent vertices.  Suppose $v_k$ is the common neighbour of $v_i$ and $v_j$.

   If  $a=2$,  there are only two positive walks of length 2 between any two positively adjacent vertices. By the same way as $b=1$, we get that there is only one positive walk of length 2 between each pair of positively adjacent vertices if $a=1$. In these two cases, we claim that $v_k\overset{-} {\sim}v_i,v_j$. Otherwise, there will be one  negative walk of length 2 between $v_i$ and $v_k$, which contradicts  both $a=2$ and $a=1$.

   Since $v_k\overset{-} {\sim}v_i,v_j$, both $v_iv_k$ and $v_jv_k$ satisfy $b=1$. So $v_k$ is not adjacent to remaining neighbours of both $v_i$ and $v_j$. Suppose $v_l$ is the remaining negative neighbour of $v_i$. There is one positive walk of length 2 between $v_i$ and $v_l$ by $b=1$. Denote the common neighbour of  $v_i$ and $v_l$ by $v_m$, then we have $v_m\overset{+} {\sim}v_i, v_l$ by the net-degree of $v_i$. But there is one negative walk of length 2 between $v_i$ and $v_m$, which contradicts  both $a=2$ and $a=1$.
\end{proof}

\begin{lemma}
If  $\dot{G}\in \mathcal{C}_1\bigcup\mathcal{C}_4\bigcup\mathcal{C}_5$ is a connected non-complete 6-regular and 0 net-regular SRSG, then $(a,b)\neq (2,0)$.
\end{lemma}

\begin{proof}
If $\dot{G}$ has $(a,b)=(2,0)$, then there are only two positive walks of length 2 between any two positively adjacent vertices and there are no walks of length 2 or two positive and two negative walks of length 2 between any two negatively adjacent vertices in $\dot{G}$ by Lemma \ref{le2} and $r=6$.

Suppose $v_iv_j$ is a negative edge of $\dot{G}$. If there are two positive and two negative walks of length 2 between  $v_i$ and $v_j$, we suppose $N(v_i)\bigcap N(v_j)=\{v_k,v_l,v_m,v_s\}$ where $v_k\overset{+} {\sim}v_i$, $v_k\overset{-} {\sim}v_j$, $v_l\overset{-} {\sim}v_i$, $v_l\overset{+} {\sim}v_j$. For vertices $v_m$ and $v_s$, at least one of them is positively adjacent to both $v_i$ and $v_j$. Suppose $v_m\overset{+} {\sim}v_i,v_j$ without loss of generality. Then there is one negative walk of length 2 between $v_i$ and $v_m$, which contradicts $a=2$. Therefore, there are no walks of length 2 between any two negatively adjacent vertices in $\dot{G}$.

Since there are only two positive walks of length 2 between any two positively adjacent vertices, these two walks are $+P_3$. Therefore, $G^+$ is the union of $K_4$, i.e., $4|n$. We have $c(n-7)=-12$ by equation  \eqref{6-0-abcn} and then $n=19,13,11,10,9$. Obviously, they cannot satisfy $4|n$.
\end{proof}

\begin{lemma}\label{6-0-ab-1}
  Let $\dot{G}\in \mathcal{C}_1\bigcup\mathcal{C}_4\bigcup\mathcal{C}_5$ be a connected non-complete 6-regular and 0 net-regular SRSG, then $(a,b)\neq (2,-4)$, $(1,-4)$, $(0,-4)$.
\end{lemma}

\begin{proof}
 Let $v_iv_j$ be a negative edge of  $\dot{G}$. Then there are four negative walks of length 2 between $v_i$ and $v_j$ if $b=-4$. Suppose that $N(v_i)=\{v_j,v_k,v_l,v_m,v_s,v_{t_1}\}$,  $N(v_j)=\{v_i,v_k,v_l,v_m,v_s,v_{t_2}\}$ and $v_k, v_l\overset{+} {\sim}v_i$,\;$ v_k, v_l\overset{-} {\sim}v_j$, $v_m, v_s\overset{-} {\sim}v_i$,\; $v_m, v_s\overset{+} {\sim}v_j$. Then $v_i\overset{+} {\sim}v_{t_1}$ and $v_j\overset{+} {\sim}v_{t_2}$. We consider the negative edge  $v_iv_m$. It must have $v_m\overset{+} {\sim} v_s$. Similarly, we have $v_k\overset{+} {\sim}v_l$ by the negative edge $v_jv_k$.

 If $(a,b)=(2,-4)$, there are only two  positive walks of length 2 between any two positively adjacent vertices in $\dot{G}$. Now, the positive edges $v_iv_k$ and $v_iv_l$ satisfy $a=2$. Then $v_k,v_l\nsim v_m,v_s,v_{t_1}$. But there are at most three common neighbours between  $v_k$ and $v_j$ for the negative edge $v_jv_k$, which is contradictory to $b=-4$.

%
%

   If $(a,b)=(1,-4)$, then there are one positive walk of length 2 or three positive  and two negative walks of length 2 between any two positively adjacent vertices. We have $v_k,v_l\overset{-} {\sim}v_{t_1}$ for the  positive edges $v_iv_k$ and $v_iv_l$. Now, there are two negative walks of length 2 between $v_{t_1}$ and $v_i$. But we get that $a\leq-3$ according to the situation of these two walks and Lemma \ref{le2}. A contradiction.

   If $(a,b)=(0,-4)$,  then there are no walks of length 2 or two positive  and two negative walks of length 2 between any two positively adjacent vertices. By the  same way as the case of $a=1$, we also have   $v_k,v_l\overset{-} {\sim}v_{t_1}$ and then $a\leq-3$. A contradiction.

\end{proof}

\begin{lemma}
  If $\dot{G}\in \mathcal{C}_1\bigcup\mathcal{C}_4\bigcup\mathcal{C}_5$ is a connected and non-complete 6-regular and 0 net-regular SRSG, then $(a,b)\neq(1,0)$, $(1,-1)$.
\end{lemma}

\begin{proof}
  By the proof of Lemma \ref{6-0-21.11}, there must be one positive walk of length 2 between two positively adjacent vertices if $a=1$. Let $\sigma$ be the sign function of $\dot{G}$ and $v_iv_j$ be an edge of  $\dot{G}$.

  For $b=0$,  there are two cases for $v_i$ and $v_j$ if $\sigma(v_iv_j)=-1$: $v_i,v_j$ don't have common neighbours; or  there are two positive and two negative walks of length 2 between $v_i$ and $v_j$. For the latter case, there must be two positively adjacent vertices with one negative walk of length 2  between them. This contradicts  $a=1$. Therefore, two negatively adjacent vertices  don't have common neighbours. Suppose that $\sigma(v_iv_j)=1$, then they have one common neighbour, denoted by $v_k$, and $v_k\overset{+} {\sim}v_i,v_j $ by $b=0$. Let $v_l$ be the remaining positive neighbour of  $v_i$, then there is no  positive walk of length 2 between  $v_l$ and  $v_i$ since $v_l\nsim v_j,v_k$. A contradiction.

  If $b=-1$, then there are two negative walks and one positive walk of length 2 between two negatively adjacent vertices.  Suppose that $\sigma(v_iv_j)=-1$ and $N(v_i)\bigcap N(v_j)=\{v_k,v_l,v_m\}$. Then we can suppose $v_k\overset{-} {\sim}v_i,v_j$,  $v_l\overset{+} {\sim}v_i$, $v_l\overset{-} {\sim}v_j$, $v_m\overset{-} {\sim}v_i$, $v_m\overset{+} {\sim}v_j$ and $v_s, v_t\in N^+(v_i)$. We consider the negative edge $v_iv_m$, then we have  $v_m\overset{-} {\sim}v_l$ or  $v_m$ is negatively adjacent to one of $v_s$ and $v_t$. If  $v_m\overset{-} {\sim}v_l$, there are two  positive walks of length 2 between $v_i$ and $v_l$. This is  contradictory to  $a=1$. If $v_m$ is negatively adjacent to one of $v_s$ and $v_t$, suppose $v_s$, then we have $v_m\overset{-} {\sim}v_k$ or $v_m\overset{+} {\sim}v_t$ or $v_m\overset{+} {\sim}v_l$. If $v_m\overset{-} {\sim}v_k$,  there are two  positive walks of length 2 between $v_i$ and $v_k$. If $v_m\overset{+} {\sim}v_t$ or $v_m\overset{+} {\sim}v_l$, there is  one negative walk of length 2 between $v_i$ and $v_t$ or $v_l$. They are contradictory to $(a,b)=(1,-1)$.
\end{proof}

 If $(a,b)=(0,-1)$, we have $c(n-7)=-3$ by equation \eqref{6-0-abcn}. So the possible parameters sets are $(10,6,0,-1,-1)$ and $(8,6,0,-1,-3)$. Since $(a,b)=(0,-1)$, two adjacent vertices in $G$ have zero or three or four common neighbours. For $c=-1$, any two non-adjacent vertices in $G$ have three common neighbours. All 6-regular graphs with order 10 (Figure \ref{10-6}) cannot satisfy these conditions. And $c=-3$ implies that any two non-adjacent vertices in $G$ have five common neighbours. There is only one  6-regular graphs with order 8, i.e. $G_8$. It is obvious that two non-adjacent vertices in $G_8$ have six common neighbours. So there is no such SRSG.

\begin{lemma}\label{6-0-ab-1}
  If $\dot{G}\in \mathcal{C}_1\bigcup\mathcal{C}_4\bigcup\mathcal{C}_5$ is a connected non-complete 6-regular and 0 net-regular SRSG, then $(a,b)\neq (-1,-1)$.
\end{lemma}

\begin{proof}
  We suppose that $(a,b)=(-1,-1)$. Then there are two negative walks and one positive walk of length 2 between any two adjacent vertices. Let $v_iv_j$ be a positive edge and $N(v_i)\bigcap N(v_j)=\{v_k,v_l,v_m\}$ and $v_s,v_t$ be the remaining two neighbours of $v_i$. By the item (7) of Lemma \ref{6-0}, the positive walk of length 2 between  $v_i$ and $v_j$ must be $-P_3$. So we suppose $v_k\overset{+} {\sim}v_i, v_k\overset{-} {\sim}v_j$, $v_l\overset{-} {\sim}v_i, v_j$, $v_m\overset{-} {\sim}v_i, v_m\overset{+} {\sim}v_j$.

  For vertices $v_s,v_t$, we suppose $v_s\overset{+} {\sim}v_i,v_t\overset{-} {\sim}v_i$. Then we get that  $v_l\overset{+} {\sim}v_m$ or $v_l\overset{+} {\sim}v_t$ by the negative edge $v_iv_l$.

  If $v_l\overset{+} {\sim}v_m$, we have $v_m\overset{-} {\sim}v_k$ or  $v_m\overset{-} {\sim}v_s$ for the negative edge $v_iv_m$. In either case, $v_m\nsim v_t$. So it must have $v_t\overset{+} {\sim}v_l$ to satisfy $b=-1$ for the negative edge $v_iv_t$. But there will be three negative walks of length 2 between $v_i$ and $v_l$. A contradiction.

  If $v_l\overset{+} {\sim}v_t$, then we have $v_l\overset{-} {\sim}v_m$ or  $v_l\overset{+} {\sim}v_k$ or  $v_l\overset{+} {\sim}v_s$ for the negative edge $v_iv_l$. In fact, $v_l\overset{-} {\sim}v_m$ will lead to two  positive walks of length 2 between $v_i$ and $v_m$. A contradiction.  For the last two cases, we have  $v_t\overset{+} {\sim}v_m$ to satisfy  $b=-1$ for the negative edge $v_iv_m$ since  $v_l\nsim v_m$. But the situation of  two negative walks of length 2 between $v_i$ and $v_t$ is contradictory to Lemma \ref{le2}.

  Therefore, $(a,b)\neq (-1,-1)$.
\end{proof}

For the case of $(a,b)=(0,0)$, we have $c(n-7)=-6$ by equation \eqref{6-0-abcn}. Then the possible parameters are: $(13,6,0,0,-1)$,  $(10,6,0,0,-2)$, $(9,6,0,0,-3)$, $(8,6,0,0,-6)$.

We consider that any two adjacent vertices don't have common neighbours firstly. Let $v_iv_j$ be an edge of $\dot{G}$. Since $a=b=0$ and $r=6$, we have $n\geq12$. Suppose $v_iv_j$ is positive and $v_k\overset{+}{\sim}v_i$. If $c=-1$, then $v_k$ has three common neighbours with $v_j$. Then $v_k$ is just adjacent to two neighbours of $v_j$. Since $a=b=0$, $v_k$ is not adjacent to other neighbours of $v_i$. Then  $n>13$. All the parameters sets are not true.

Next, we concern the cases of  $\dot{G}$ contains unbalanced triangles and   $\dot{G}$ contains only balanced triangles, respectively. There are two kinds of unbalanced triangles and we give a result about the  first type of unbalanced triangles as follows.

 \begin{lemma}\label{6-0-a0b0}
  Let $\dot{G}\in \mathcal{C}_1\bigcup\mathcal{C}_4\bigcup\mathcal{C}_5$ be a connected non-complete 6-regular and 0 net-regular SRSG. If $a=b=0$ and it contains an unbalanced triangle of the first type, then  $G$ is isomorphic to $G_8$ in Figure \ref{8-6.864-4-6}.
\end{lemma}

\begin{proof}
  Let $v_iv_jv_k$ be an unbalanced triangle of the first type in $\dot{G}$ with the negative edge $v_jv_k$. If $a=b=0$, then $v_i,v_j$ have four common neighbours. Suppose that $N(v_i)=\{v_j,v_k,v_l,v_m,v_s, v_{t_1}\}$ and $N(v_j)=\{v_i,v_k,v_l,v_m,v_s, v_{t_2}\}$. And we suppose $v_l\overset{-} {\sim}v_i$ and $v_l\overset{+} {\sim}v_j$ by Lemma \ref{le2}.  There are two cases for $v_m,v_s$: One vertex is positively adjacent to $v_i,v_j$, the other one is negatively adjacent to $v_i,v_j$; Both $v_m$ and $v_s$ are negatively adjacent to $v_i,v_j$. In fact, the first case is not true. Suppose $v_m\overset{+} {\sim}v_i,v_j$ and  $v_s\overset{-} {\sim}v_i,v_j$. Then we consider the positive edge $v_iv_m$. Since $a=0$, it must have $v_m\overset{-} {\sim}v_k$. But we deduce that $a\leq -3$ from the positive edge $v_iv_k$ by Lemma \ref{le2}. A contradiction. So we have $v_m, v_s\overset{+} {\sim}v_i,v_j$.

  Now we claim that $v_{t_1}$ must have common neighbours with $v_i$. Otherwise, if $N(v_i)\bigcap N(v_{t_1})=\emptyset$, then  $v_{t_1}\nsim v_k,v_l,v_m,v_s$. We consider the negative edge $v_iv_l$, then it must have $v_l\overset{+} {\sim}v_k$ and $v_l$ is positively adjacent to one of $v_m,v_s$ and negatively adjacent to the remaining one. Suppose $v_l\overset{+} {\sim}v_m$ and $v_l\overset{-} {\sim}v_s$ without loss of generality. Then we consider the negative edge $v_iv_m$. It must have $v_m\overset{+} {\sim}v_k$ and $v_m\overset{-} {\sim}v_s$ since $v_m\nsim v_{t_1}$. But the negative degree of $v_s$ is 4 now. A contradiction. The case of $v_{t_2}$ is similar. Then it must have $v_{t_1},  v_{t_2}\sim v_k,v_l,v_m,v_s$ by $a=b=0$.

  Now we consider the  edge $v_iv_l$, there are three cases: Case 1. $v_l\sim v_k,v_m$;   Case 2. $v_l\sim v_k,v_s$;  Case 3. $v_l\sim v_m,v_s$.

 For Case 1, we consider the  edge $v_iv_k$. Then we have  $v_k\sim v_m$ or $v_k\sim v_s$. If $v_k\sim v_m$,  we have $d(v_k)=d(v_l)=d(v_m)=6$. But there are only two common neighbours for $v_i$, $v_s$.   A contradiction.  If $v_k\sim v_s$, we have $v_m\sim v_s$ by the edge $v_iv_m$. Now there are only three common neighbours for $v_k$ and $v_{t_1}$ and we have $d(v_k)=d(v_l)=d(v_m)=d(v_s)=d(v_i)=d(v_j)=6$. So it must have $v_{t_1}\sim  v_{t_2} $. In this case, we get that $G$ is isomorphic to $G_8$.

 The Case 2 and Case 3 are similar as Case 1, they also deduce that $G$ is isomorphic to $G_8$.

%

\end{proof}

\begin{figure}
  \centering
  \subfigure[$G_8$]{
  \includegraphics[width=5cm]{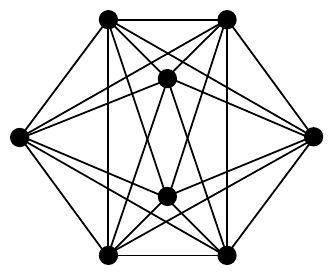}}
  \subfigure[$\dot{S}^4_8$]{
    \includegraphics[width=5cm]{864-4-6.pdf}}
  \caption{$G_8$ and $\dot{S}^4_8$}\label{8-6.864-4-6}
\end{figure}

 If $(a,b)=(0,0)$, we only need to consider the set of  $(8,6,0,0,-6)$ by the above Lemma.   Since $a=b=0$ and we consider the case of that there are two adjacent vertices which have common neighbours, then there are two positive and two negative walks of length 2 between these two adjacent vertices.  Let $v_1v_2$ be a positive edge and $N(v_1)\bigcap N(v_2)=\{v_3,v_4,v_5,v_6\}$ and $v_7,v_8$ be the remaining two neighbours of $v_1$ and $v_2$, respectively. Since $G$ is isomorphic to $G_8$, it must have $v_7,v_8\sim v_3,v_4,v_5,v_6$. And there are two cases for the walks between  $v_1$ and $v_2$:

 Case 1. $v_3\overset{+} {\sim}v_1,v_2$,  $v_4\overset{+} {\sim}v_1, v_4\overset{-} {\sim}v_2$, $v_5\overset{-} {\sim}v_1, v_5\overset{+} {\sim}v_2$, $v_6\overset{-} {\sim}v_1,v_2$;

  Case 2. $v_3, v_6\overset{-} {\sim}v_1,v_2$, $v_4\overset{+} {\sim}v_1, v_4\overset{-} {\sim}v_2$, $v_5\overset{-} {\sim}v_1, v_5\overset{+} {\sim}v_2$.

  For Case 1, we have $v_7\overset{-} {\sim}v_1,v_8\overset{-} {\sim}v_2$. We consider the negative edge  $v_1v_6$, then we have $v_5\overset{+} {\sim}v_6$ or $v_7\overset{+} {\sim}v_6$. If
 $v_5\overset{+} {\sim}v_6$, we consider the negative edge  $v_1v_5$. Then it must have $v_5\overset{-} {\sim}v_7$ since $v_7\sim v_5$. We continue consider the negative edge  $v_1v_7$, then we get that $v_7\overset{+} {\sim}v_6$. But this is contradictory to $b=0$ for the negative edge  $v_1v_6$. By the similar way, we get that $v_7\overset{-} {\sim}v_5$ and so $v_5\overset{+} {\sim}v_6$  if  $v_7\overset{+} {\sim}v_6$. Also a contradiction. For case 2, we have $v_7\overset{+} {\sim}v_1,v_8\overset{+} {\sim}v_2$. And we have $v_4\overset{-} {\sim}v_7$ for the positive edge  $v_1v_7$. But the negative walks of length 2 between $v_1$ and $v_4$ are contradictory to Lemma \ref{le2}.

 Therefore, $(a,b)\neq(0,0)$ in this case.

Now, we consider that $\dot{G}$ contains an unbalanced triangle of the second type. And we show that $b\neq0$ by the following Lemma.
\begin{lemma}
If  $\dot{G}\in \mathcal{C}_1\bigcup\mathcal{C}_4\bigcup\mathcal{C}_5$ is a connected non-complete 6-regular and 0 net-regular SRSG and it has an unbalanced triangle of the second type, then $b> 0$.
\end{lemma}
\begin{proof}
  Let $v_iv_jv_k$ be an unbalanced triangle of the second type in $\dot{G}$. Then $v_i\overset{-} {\sim}v_j,v_k$ and $v_j\overset{-} {\sim}v_k$. There is one positive walk of length 2 between $v_i$ and $v_j$. Then we can get  $b\geq -1$ directly by Lemma \ref{le2}. If  $b= -1$ or 0, then there must be two negative walks of length 2 between $v_i$ and $v_j$. So we can suppose that $v_l$ and $v_m$ are the common neighbours of $v_i,v_j$ and $v_l\overset{+} {\sim}v_i$, $v_l\overset{-} {\sim}v_j$, $v_m\overset{-} {\sim}v_i$ and $v_m\overset{+} {\sim}v_j$. For the negative edge $v_iv_k$, there must be two negative walks of length 2 between $v_i$ and $v_k$. So it must have $v_m\overset{+} {\sim}v_k$. But there will be four negative walks of length 2 between $v_i$ and $v_m$ by Lemma \ref{le2}, which leads to  $b\leq -3$.  A contradiction. Therefore, $b>0$.

\end{proof}

Finally, we consider  that $\dot{G}$ only contains balanced triangles. Then  there are only positive walks of length 2 between any two positively adjacent vertices and only negative walks of length 2 between any two negatively adjacent vertices. Then $a=b=0$ is impossible.

%
%
%
%
%
%

The following conclusions can be drawn by summarizing the contents of this subsection.
\begin{theorem}
  There are four connected 6-regular and 0 net-regular SRSGs: $\pm \dot{S}^4_{8}$ and $\pm \dot{S}_{16}$.
\end{theorem}

\vspace*{2mm}

\textbf{ Declaration of competing interest}

The authors declare that they have no conflict of interests.

\vskip 0.6 true cm
{\textbf{Acknowledgments}}

 This research is supported by the National Natural Science Foundation of China (Nos. 12331012,11971164).
\baselineskip=0.25in

\linespread{1.10}

\begin{appendices}
\section{}

\begin{figure}[h]
  \centering
  \includegraphics[width=12cm]{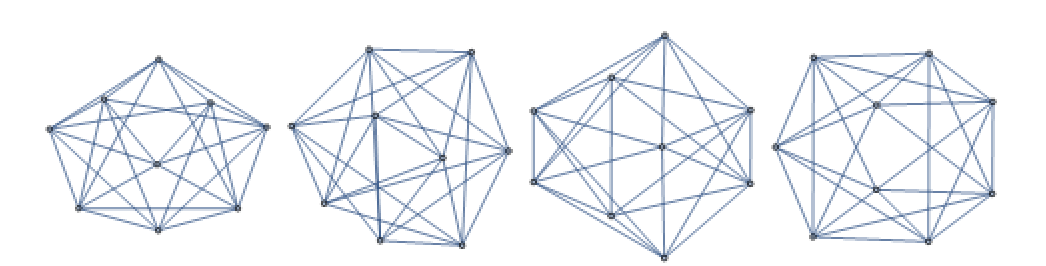}
  \caption{All 6-regular graphs with order 9}\label{9-6}
\end{figure}

\begin{figure}[h]
  \centering
  \includegraphics[width=12cm]{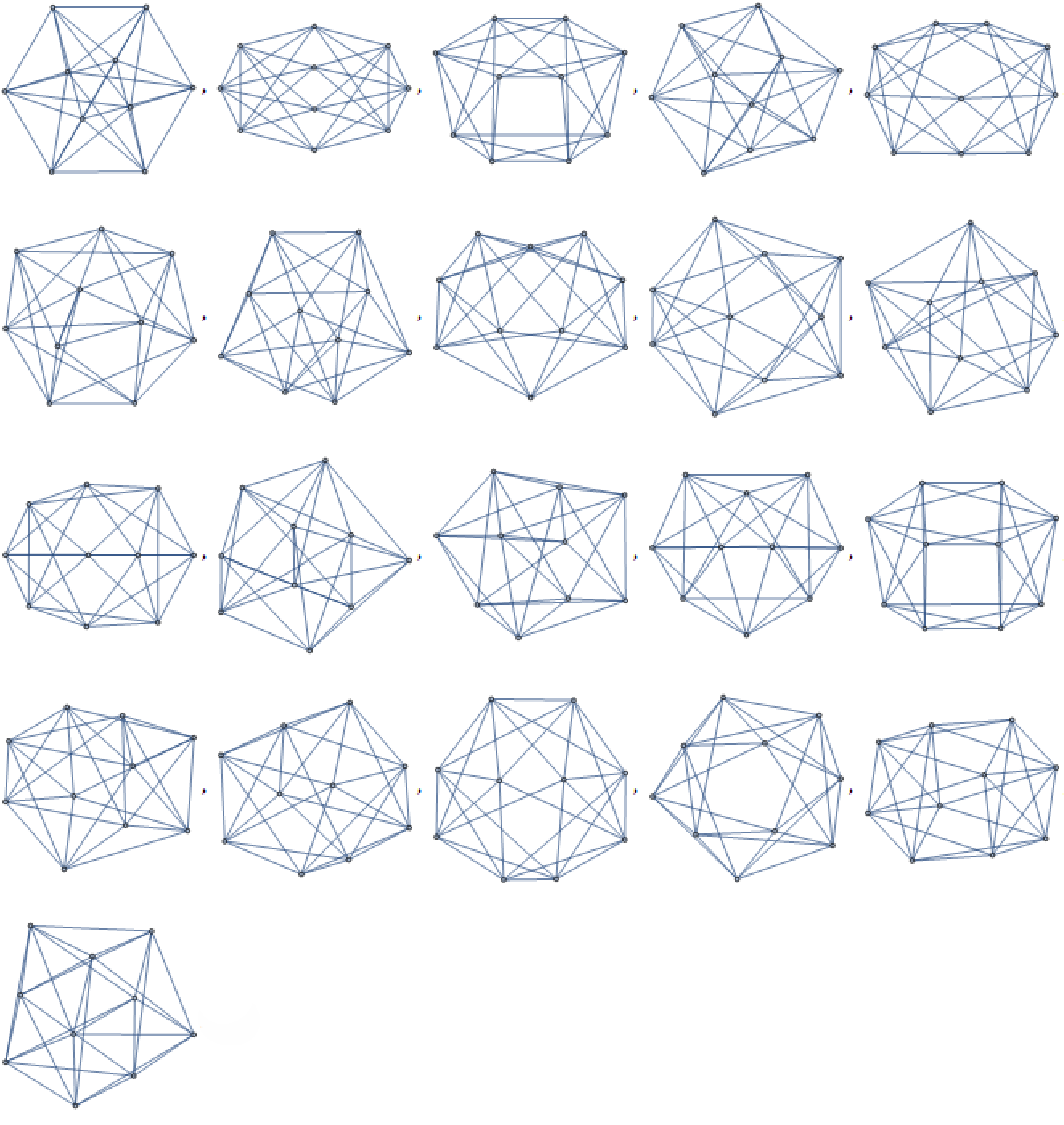}
  \caption{All 6-regular graphs with order 10}\label{10-6}
\end{figure}

\end{appendices}

\end{document}